\documentclass [11pt,reqno]{amsart}
\usepackage {amsmath,amssymb,epsfig,verbatim,geometry}
\usepackage[all]{xy}

\newif\ifpdf
\ifpdf
  \usepackage[pdftex]{graphicx}
  \usepackage[pdftex]{hyperref}
\else
  \usepackage{graphicx}
\fi
\numberwithin{equation}{section}       
 \theoremstyle{plain}    
 \newtheorem{thm}{Theorem}[section]
 \numberwithin{equation}{section} 
 \numberwithin{figure}{section} 
 \theoremstyle{plain}
 \theoremstyle{plain}    
 \newtheorem{cor}[thm]{Corollary} 
 \theoremstyle{plain}    
 \newtheorem{prop}[thm]{Proposition} 
 \theoremstyle{plain}    
 \newtheorem{lem}[thm]{Lemma} 
 \theoremstyle{remark}
 \newtheorem{rem}[thm]{Remark}
 \theoremstyle{definition}
\newtheorem{exa}[thm]{Example}
\theoremstyle{plain}

\theoremstyle{definition}
\newtheorem{defi}[thm]{Definition}

\newtheorem*{ackn}{Acknowledgements}

\newcommand{\C}{{\mathbb{C}}}
\newcommand{\N}{{\mathbb{N}}}
\newcommand{\PP}{{\mathbb{P}}}
\newcommand{\Q}{{\mathbb{Q}}}
\newcommand{\R}{{\mathbb{R}}}

\newcommand{\ie}{i.e.~}

\newcommand{\psh}{{\mathrm{PSH}}}

\newcommand{\tX}{\widetilde{X}}

\renewcommand{\a}{\alpha}
\renewcommand{\b}{\beta}

\newcommand{\e}{\varepsilon}
\newcommand{\om}{\omega}
\newcommand{\f}{\varphi}

\newcommand{\p}{\psi}

\newcommand{\D}{\Delta}

\newcommand{\Amp}{\mathrm{Amp}\,}

\newcommand{\Ric}{\mathrm{Ric}}

\newcommand{\reg}{\mathrm{reg}}

\newcommand{\pddt}{\frac{\partial}{\partial t}}

\newcommand{\vol}{\operatorname{vol}}
\newcommand{\tr}{\operatorname{tr}}

\begin{document}

\setcounter{tocdepth}{1}

\title{Regularizing properties of the twisted K\"ahler-Ricci flow}

\date{\today}

\author{Vincent Guedj, Ahmed Zeriahi}

\address{Institut Universitaire de France  \& Institut Math\'ematiques de Toulouse, \\
Universit{\'e} Paul Sabatier\\ 31062 Toulouse cedex 09\\ France}

\email{vincent.guedj@math.univ-toulouse.fr}

\address{Institut de Math\'ematiques de Toulouse,   \\ Universit\'e Paul Sabatier \\
118 route de Narbonne \\
F-31062 Toulouse cedex 09\\}

\email{ahmed.zeriahi@math.univ-toulouse.fr}

\begin{abstract}  
Let $X$ be a compact K\"ahler manifold. We show that the K\"ahler-Ricci flow (as well as its twisted versions) can be run from an arbitrary positive closed current with zero Lelong numbers and immediately smoothes it.
\end{abstract} 

\maketitle

\tableofcontents

\newpage

\section*{Introduction}
  
Let $X$ be a compact K\"ahler manifold of complex dimension $n$ 
and $\a_0 \in H^{1,1}(X,\R)$  a K\"ahler class.
The purpose of this note is to show that the K\"ahler-Ricci flow
$$
\frac{\partial \omega_t}{\partial t}=-Ric(\omega_t)
$$
can be run from an initial data $T_0 \in \alpha_0$ which is an arbitrary positive closed current with zero Lelong numbers, i.e.
there is a family of K\"ahler forms $(\omega_t)_{t>0}$ solutions of the above equation,
such that $\omega_t \rightarrow T_0$ as $t \rightarrow 0^+$.

We shall actually consider slightly more general twisted K\"ahler-Ricci flows
$$
\frac{\partial \omega_t}{\partial t}=-Ric(\omega_t)+\eta,
$$
where $\eta$ is a fixed closed $(1,1)$-form. The case when $\eta=[D]$ is the current of integration along an hypersurface has become quite important recently in connection with the K\"ahler-Einstein problem on Fano manifolds (see e.g. \cite{Don,Ber,CDS,Tian12}), however we will restrict here to the case of a smooth form $\eta$.

\smallskip

It is standard \cite{Cao85, Tsu, TZha} that when $T_0$ is a K\"ahler form, such a flow admits a unique solution on a maximal interval of time $[0,T_{max}[$, where
$$
T_{max}:=\sup \left\{ t \geq 0 \, | \, t K_X+t \{\eta\} +\alpha_0   \text{ is nef}\right\}.
$$
Our main result is the following:

\medskip

\noindent {\bf Theorem A.}
{\it Let $T_0 \in \a_0$ be a positive current with zero Lelong numbers. There exists a unique 
maximal family
$(\omega_t)_{0<t<T_{max}}$ of K\"ahler forms such that
$$
\frac{\partial \omega_t}{\partial t}=-Ric(\omega_t)+\eta,
$$
and $\omega_t$ weakly converges towards $T_0$, as $t \searrow 0^+$.
}
\medskip

We shall prove this result by working at the level of potentials and establishing  smoothing properties
of more general parabolic complex  Monge-Amp\`ere  flows (see Theorem \ref{thm:main}).
We give precise information on the weak continuity at time zero, depending on the properties
of the initial potential (convergence in energy or in capacity).

 Smoothing properties of the K\"ahler-Ricci flow have been known and used for a long time (see e.g.
\cite{BM87,Tian97,PSSW08}). Attempts to run the K\"ahler-Ricci flow from a degenerate initial data have
motivated several recent works \cite{CD07,CT08,CTZ11,ST,SzTo}. The best results so far were obtained in 
\cite{ST}, where the authors succeeded in running the K\"ahler-Ricci flow from an initial
current $T_0$ with continuous potentials.

\medskip

 Starting from an initial data $T_0$ having positive Lelong numbers at some points is an interesting issue
and we shall discuss it on our way to proving the main result as well as in section \ref{sec:concluding},
where we show that the normalized K\"ahler-Ricci flow eventually smoothes out any arbitrary positive closed current on a manifold with non-positive first Chern class (see Theorem \ref{thm:cao}). This is
however not necessarily the case on a Fano manifold (see Example \ref{exa:pn}), although the smoothing property  could be useful in analyzing the long-term behavior of  the normalized K\"ahler-Ricci flow
on Fano manifolds (see Theorem \ref{thm:fano} and Remark \ref{rem:heuristic}).

\medskip

 We push  our analysis further in section \ref{sec:last} and show that one can start the twisted K\"ahler-Ricci flow from a positive current representing a nef class (Theorem \ref{thm:nef}), and then
treat the case of mildly singular varieties. This is a particularly important situation for applications, in connection with the Minimal Model Program. Our main result extends to this context as follows:

\medskip

\noindent {\bf Theorem B.}
{\it
Let $X$ be a projective complex variety with log terminal singularities. Let $T_0$
be a positive $(1,1)$-current with zero Lelong numbers representing a K\"ahler class 
$\a_0 \in H^{1,1}(X,\R)$. Then there exists a 
continuous family $(\om_t)_{t\in[0,T_{max}[}$ of positive $(1,1)$-currents such that
\begin{itemize} 
\item[(i)] $[\om_t]=\a_0-tc_1(X)$ in $H^{1,1}(X,\R)$; 
\item[(ii)] $\omega_t \rightarrow T_0$ as $t \rightarrow 0$;
\item[(iii)] $(\om_t)_{t\in(0,+\infty)}$ restricts to a smooth path of K\"ahler forms on $X_\reg$ satisfying
$$
\frac{\partial\om_t}{\partial t}=-\Ric(\om_t). 
$$
\end{itemize}
}
\medskip

\begin{ackn} 
This work is a natural follow-up to the lecture notes \cite{BG13} and we are grateful to 
S\'ebastien Boucksom for numerous discussions.
We also thank Eleonora DiNezza, Hoang Chinh Lu and Philippe Lauren\c cot for useful comments on
a preliminary draft.
\end{ackn}

\section{Strategy of the proof}

\subsection{Reduction to a scalar Monge-Amp\`ere flow}

Fix $\omega$ a reference K\"ahler form in the initial K\"ahler class $\alpha_0$.
Thus $T_0=\omega+dd^c \f_0$ for some $\omega$-plurisubharmonic function $\f_0$, by the 
$\partial\overline{\partial}$-lemma. We use here and in the sequel the standard normalization
$$
d=\partial+\overline{\partial}, \; d^c=\frac{1}{2 i \pi} (\partial-\overline{\partial}),
\text{ so that }
dd^c =\frac{i}{\pi} \partial\overline{\partial}.
$$
Set 
$$
\theta_t:=\omega+t\eta-t Ric(\omega)
$$
 and consider the parabolic scalar flow
$$
 \frac{\partial \f_t}{\partial t}=\log \left[ \frac{(\theta_t+dd^c \f_t)^n}{\omega^n} \right]
$$
with ${\f_t}_{|t=0}=\f_0$. Let ${\omega}_t$ denote the K\"ahler form 
${\omega}_t:=\theta_t+dd^c \f_t$ and observe that
$$
\frac{\partial {\omega}_t}{\partial t}=\eta-Ric(\omega)+dd^c \dot{\f_t}=-Ric({\omega}_t)+\eta,
$$
since $dd^c \log({\omega}_t^n/\omega^n)=-Ric({\omega}_t)+Ric(\omega)$. 

Conversely, one easily shows that if $\omega_t$ evolves along the twisted K\"ahler-Ricci flow, 
then $\omega_t=\theta_t+dd^c \f_t$ where $\f_t$ satisfies the above scalar parabolic flow, up to a 
time dependent additive constant. We shall normalize the latter to be zero.

\subsection{Approximation process}

Fix $\f_0 \in PSH(X,\omega)$ an arbitrary $\omega$-psh function.
We are going to approximate  $\f_0$ by a decreasing sequence
$\f_{0,j}$ of smooth strictly $\omega$-psh functions 
$\f_{0,j} \in PSH(X,\omega) \cap {\mathcal C}^{\infty}(X)$.
This is always possible thanks to a regularizing result of Demailly \cite{Dem92}.

We consider the corresponding solution $\f_{t,j} \in PSH(X,\theta_t) \cap {\mathcal C}^{\infty}(X) $ to
the above scalar parabolic flow. By the discussion above, these flows are well defined on 
$[0,T_{max}[ \times X$.

Our goal is then to establish various a priori estimates which will allow us
to pass to the limit as $j \rightarrow +\infty$. 
For example, when $\f_0$ is {\it bounded}  we are going to prove that for each $\e>0$ 
and $0<T<T_{max}$ fixed,
\begin{enumerate}
\item $(x,t,j) \mapsto \f_{t,j}(x)$ is uniformly bounded on $X \times [0,T] \times \N$;
\item  $(x,t,j) \mapsto \dot{\f}_{t,j}(x)$ is uniformly bounded on $X \times [\e,T] \times \N$;
\item $(x,t,j) \mapsto \Delta \f_{t,j}(x)$ is uniformly bounded on $X \times [\e,T] \times \N$;
\end{enumerate}
Here $\Delta$ denotes the Laplace operator with respect to a fixed metric, e.g. $\omega$.

Thanks to the complex parabolic Evans-Krylov theory and Schauder estimates (see \cite{ShW11} for a recent account in the K\"ahler-Ricci flow context), these bounds
allow to show that $\f_{t,j} \rightarrow \f_t$ in ${\mathcal C}^{\infty}(X \times ]0,T])$, as 
$j \rightarrow +\infty$. We'll then check that $\f_t \rightarrow \f_0$ as $t \rightarrow 0^+$. This is
obvious (by global continuity) if $\f_0$ is {\it continuous}, slightly more involved when 
$\f_0$ is less regular:
\begin{itemize}
\item when $\f_0$ is bounded, we show   that
$\f_t$ converges to $\f_0$ in capacity as $t \rightarrow 0$;
\item when  $\f_0$ has finite energy, we show that 
the approximants $\f_{t,j}$ have uniformly bounded energies, hence are relatively compact
in the finite energy class ${\mathcal E}^1(X,2\omega)$. The $\f_{t,j}$'s then form  a compact
family in ${\mathcal C}^{\infty}(X \times ]0,T])$ and converge in energy towards
$\f_0$ as $t \rightarrow 0$ and $j \rightarrow +\infty$;
\item for arbitrarily singular initial potential $\f_0 \in L^1(X)$ we use the 
convexity property of the mean value $t \mapsto V^{-1} \int_X \f_{t,j} \, d\mu$ to control 
$\sup_X \f_{t,j}$ from below and show continuity in the $L^1$-topology at time zero.
\end{itemize}

\subsection{Notations}

In the sequel we set $\chi=\eta-\Ric(\omega)$ so that
$$
\theta_t=\omega+t \chi.
$$
To simplify notations we always assume that for $0<t \leq T<T_{max}$, one has
$$
\frac{\omega}{2} \leq \theta_t \leq 2 \omega.
$$
Thus $\f_t \in PSH(X,2 \omega)$ for all $t$. We also set
$$
V:=\vol_{\omega}(X)=\int_X \omega^n=\a_0^n.
$$

We let $(CMAF)$ denote the scalar parabolic flow
$$
\hskip-4cm
(CMAF) \hskip3cm \frac{\partial \f_t}{\partial t}=\log \left[ \frac{(\theta_t+dd^c \f_t)^n}{\mu} \right]
$$
where $\mu=e^h \omega^n$ is a smooth positive measure and $h \in {\mathcal C}^{\infty}(X,\R)$
is normalized so that 
$$
V=\int_X e^h \, \omega^n.
$$

\medskip

The next four sections are devoted to proving the following:

\begin{thm} \label{thm:main}
Let $\f_0$ be an $\omega$-psh function with zero Lelong numbers. There exists a unique maximal family of
smooth strictly $\theta_t$-psh functions $(\f_t)$ such that
$$
 \frac{\partial \f_t}{\partial t}=\log \left[ \frac{(\theta_t+dd^c \f_t)^n}{\mu} \right]
$$
in $]0,T_{max}[ \times X$, with $\f_t \rightarrow \f_0$ in $L^1(X)$, as $t \searrow 0^+$.
Moreover 
\begin{itemize}
\item $\f_t$ converges in energy towards $\f_0$ if $\f_0 \in {\mathcal E}^1(X,\omega)$ has finite energy;
\item $\f_t$ is uniformly bounded and converges to $\f_0$ in capacity if $\f_0 \in L^\infty(X)$.
\end{itemize}
\end{thm}

Since the weak convergence of positive currents is equivalent to the $L^1-$convergence of 
(normalized) potentials, this result clearly contains our Main Theorem as a particular case.

The uniqueness property has to be understood in the following weak sense: if $\p_t$ is another solution of the parabolic scalar flow which converges in $L^1$ to the same initial data $\f_0$, 
then $\p_t$ lies below $\f_t$, i.e. $\f_t$ is the envelope of such solutions
(see \cite{Top10,GT11} for the related notion of unique "maximally stretched" solution).

When the Monge-Amp\`ere measure $MA(\f_0)$ is absolutely continuous with respect to Lebesgue measure, with density $f_0 \in L^p$, $p>1$ and continuous initial potential $\f_0$, 
it has been shown in \cite{ST} that there is no other solution $\p_t$ (this follows easily in this case from the maximum principle).

We discuss in section \ref{sec:concluding} how one can try and run the (normalized or twisted) 
K\"ahler-Ricci flow from a positive current having positive Lelong numbers.

\section{Bounds on $\f_t$}

In this section we assume  that $\f_t$ satisfies $(CMAF)$ with an initial data $\f_0$, which
is a smooth strictly $\omega$-psh function.

\subsection{Maximum principle}

The following maximum principle is a  basic tool to establish upper and lower bounds in the sequel.

\begin{prop}\label{prop:max} 
Let $\Omega_t$ be a smooth family of K\"ahler metrics on $X$, and denote by $\D_t$ the Laplacian with respect to $\Omega_t$. 
Assume that $H\in C^\infty(X\times[0,T])$ satisfies
$$
\left(\pddt-\D_t\right)H \le 0
\; \; \text{ or } \; \; 
\dot{H_t} \le\log\left[\frac{\left(\Omega_t+dd^c H_t\right)^n}{\Omega_t^n}\right],
$$
 Then $\sup_X H_t \le\sup_X H_0$ forall $t\in[0,T]$. 

\smallskip

If we replace $\le$ with $\ge$ 
then $\inf_X H_t \ge \inf_X H_0$.
\end{prop}

We include a proof for the reader's convenience.

\begin{proof} 
Replacing $H$ with $H-\e t$ with $\e>0$,  we may assume in each case that the inequality is strict. 
By compactness $H$ achieves its supremum at some point $(x_0,t_0)\in X\times[0,T]$, and the strict differential inequality implies that $t_0$ is necessarily $0$, since otherwise we would have 
$\pddt H \ge 0$ and 
$dd^c H \le 0$ at $(x_0,t_0)$.
\end{proof}

\begin{cor} \label{cor:envelope}
Let $u_t$ (resp. $v_t$) be a subsolution (resp. a supersolution) of $(CMAF)$, i.e.
$u_0 \leq \f_0 \leq v_0$ with
$$
\dot{u_t} \le\log\left[\frac{\left(\theta_t+dd^c u_t\right)^n}{\mu}\right]
\; \;  \text{ while }  \; \;
\dot{v_t} \ge\log\left[\frac{\left(\theta_t+dd^c v_t\right)^n}{\mu}\right].
$$
Then $u_t \leq v_t$.
Thus if $\f_t,\p_t$ are solutions of $(CMAF)$ with initial data $\f_0,\p_0$, 
$$
\inf_X (\f_0-\p_0) \leq \f_t-\p_t \leq \sup_X (\f_0-\p_0).
$$
In particular if $\f_0 \leq \p_0$, then $\f_t \leq \p_t$ for all $t$.
\end{cor}

\begin{proof}
Apply  Proposition \ref{prop:max} by setting $\Omega_t=\theta_t+dd^c v_t$ and $H=u_t-v_t$.
\end{proof}

\subsection{Bounding $\f_t$ from above}

\begin{lem} \label{lem:majosup}
The function $t \mapsto \sup_X \f_t$ is quasi-decreasing, more precisely
$$
\f_t \leq \sup_X \f_0+ [ n \log 2-\inf_X h ] t.
$$
\end{lem}

\begin{proof}
Observe that $\p_t:=\sup_X \f_0+ [ n \log 2-\inf_X h ] t$ is a super-solution, i.e.
$$
\dot{\p_t} \geq \log \left[ \frac{(\theta_t+dd^c \p_t)^n}{\mu} \right]
$$
with $\p_0 \geq \f_0$ and
apply Corollary \ref{cor:envelope} to conclude.
\end{proof}

\begin{rem}
It might be useful to notice for other applications that one can get an upper-bound which is independent of $\inf_X h$. Indeed let $\p_0$ denote the solution of the elliptic problem
$$
(2 \omega+dd^c \p_0)^n=2^n \mu,
$$
normalized by $\sup_X (\f_0-\p_0)=0$. Observe that $\f_t$ is then a subsolution of the
corresponding parabolic problem
$$
\dot{\p}_t=\log\left[\frac{\left( 2 \omega+dd^c \p_t\right)^n}{\mu}\right],
\; \; {\p_t}_{|t=0}=\p_0,
$$
whose solution is $\p_t=\p_0+n t \log 2$. The comparison principle thus yields
$$
\f_t \leq \p_t=\p_0+n t \log 2.
$$
\end{rem}

An alternative observation which will reveal also useful is the following:

\begin{lem} \label{lem:convex}
The mean value $I(t)=\frac{1}{V} \int_X \f_t \, d\mu$ is quasi-decreasing, namely
$$
t \mapsto I(t)-t \log(2^n)
\text{ is non-increasing}.
$$

The function $I$ is moreover convex when $\chi \geq 0$.
\end{lem}

\begin{proof} 
It follows from the concavity of the logarithm that
$$
I'(t) =\int_X \log \left[ \frac{(\theta_t+dd^c \f_t)^n}{\mu} \right] \frac{d \mu}{V}
\leq \log \int_X \frac{(\theta_t+dd^c \f_t)^n}{V} \leq n \log 2,
$$
since we impose $\theta_t \leq 2 \omega$.

For the second assertion, observe that
$$
\ddot{\f_t}=\Delta_{\omega_t} \dot{\f_t}+\tr_{\omega_t}(\chi) \geq \Delta_{\omega_t} \dot{\f_t}
$$
when $\chi \geq 0$, hence
$$
V \, I''(t)=\int_X \ddot{\f_t} \, d\mu \geq \int_X \Delta_{\omega_t} \dot{\f_t} e^{-\dot{\f_t}} \omega_t^n
=n \int_X e^{-\dot{\f}_t} d \dot{\f_t}   \wedge d^c \dot{\f_t} \wedge \omega_t^{n-1} \geq 0.
$$
\end{proof}

Recall \cite[Proposition 1.7]{GZ05} that since $\mu$ is a smooth measure, there exists $C_{\mu}>0$ such that
$$
\sup_X \f -C_{\mu} \leq \frac{1}{V} \int_X \f \, d\mu \leq \sup_X \f
$$
for all $2\omega$-psh functions $\f$.

\subsection{Various bounds from below}

\subsubsection{Bounded initial data}
 
Recall (Corollary \ref{cor:envelope})  that $\f_t$ dominates
any { subsolution}. An easy computation yields the following:

\begin{lem}  \label{lem:minoinf}
Fix $0<T<T_{max}^2$. There exists $C=\sup_X h+C_n>0$ such that
$\p_t=(1-\sqrt{t})(\f_0-\inf_X \f_0+1)-Ct+\inf_X \f_0-1$ is a subsolution, hence
$$
(1-\sqrt{t}) (\f_0-\inf_X \f_0+1)-Ct+\inf_X \f_0-1 \leq \f_t, 
$$
$\text{ for all } (t,x) \in [0,T] \times X$.
In particular $\f_t \geq \inf_X \f_0-C'$.
\end{lem}

Such bound from below is useful to establish the "continuity" of the flow at time zero, starting from a 
bounded but non-continuous initial data.

\begin{proof}
 Observe first that $\p_t$ is $\theta_t$-plurisubharmonic, with
$$
\theta_t+dd^c \p_t=(1-\sqrt{t}) (\omega+dd^c \f_0)+\sqrt{t} \theta_{\sqrt{t}}
\geq \frac{\sqrt{t}}{2} \omega.
$$
We infer
$$
\frac{n}{2} \log t -n \log 2-\sup_X h \leq \log \left[ \frac{(\theta_t+dd^c \p_t)^n}{\mu} \right]
$$
while $\dot{\p}_t=-\frac{1}{2\sqrt{t}} (\f_0-\inf_X \f_0+1)-C \leq -\frac{1}{2\sqrt{t}}-C$.

Since $\p_t$ and $\f_0$ coincide at time zero,
the conclusion follows by adjusting the value of $C$ so that for all $t>0$,
$$
-\frac{1}{2\sqrt{t}}-C < \frac{n}{2} \log t -n \log 2-\sup_X h.
$$
\end{proof}

\subsubsection{Finite energy condition}

 Set
$$
E(\f_t):=\frac{1}{(n+1)V} \sum_{j=0}^n \int_X \f_t (\theta_t+dd^c \f_t)^j \wedge \theta_t^{n-j}.
$$
When $\chi \equiv 0$, i.e. $\theta_t \equiv \omega$, this is the Aubin-Yau energy functional which plays a crucial role in studying the K\"ahler-Einstein equation (see \cite{Tian,BEGZ} 
for recent developments). In particular
$$
E(\f_0)=\frac{1}{(n+1)V} \sum_{j=0}^n \int_X \f_0 (\omega+dd^c \f_0)^j \wedge \omega^{n-j}.
$$

\begin{defi}
We let ${\mathcal E}^1(X,\omega)$ denote the "finite energy class", i.e. the set of 
$\omega$-plurisubharmonic  functions $\f_0$ such that $E(\f_0)>-\infty$.
\end{defi}

A basic observation is the following monotonicity property:

\begin{lem} \label{lem:energymonotone}
Fix $0<T<T_{max}$.
There exists $C \geq 0$ such that
$$
t \mapsto E(\f_t) +C  t \text{ is increasing on } [0,T],
$$
with $C=0$ if $\chi=0$.
In particular 
$$
\frac{E(\f_0)-Ct}{2^n} \leq \frac{E(\f_0)-Ct}{\{ \theta_t\}^n/V}  \leq  \sup_X \f_t.
$$
\end{lem}
 
\begin{proof}
We let the reader check that 
$$
V\frac{d { E}(\f_t)}{dt}=\int\dot{\f_t} \, \omega_t^n
+\frac{1}{(n+1)} \sum_{j=0}^n \int \f_t \chi \wedge [j \theta_t +(n-j) \omega_t] 
\wedge (\omega_t)^{j-1} \wedge \theta_t^{n-j-1}.
$$

The first term is almost non-negative as  follows from the concavity of the logarithm. Namely
$$
\int \log \left( \frac{\omega_t^n}{\mu} \right) \, \frac{\omega_t^n}{V_t}
\geq -\log ({V}/{V_t}) \geq -n \log 2,
$$
where $V_t:=\int_X \omega_t^n \geq V/2^n$ by our assumption $\theta_t \geq \omega/2$.
Note that  this first term is non-negative when $\chi=0$, since $V_t=V$ in this case.

The second one (which is zero if $\chi=0$) is bounded below by $-C$ when $\chi \leq 0$,
as can be checked by writing $\chi \f_t \geq \chi\sup_X \f_t$ and estimating the remaining cup-products.

If $\chi$ is not negative, we can nevertheless find $A>0$ such that $\chi-A \theta_t \leq 0$ for
all $0 \leq t \leq T$. Observing that
$$
\int \f_t (\theta_t+dd^c \f_t)^{j} \wedge \theta_t^{n-j}
\leq \int \f_t (\theta_t+dd^c \f_t)^{j-1} \wedge \theta_t^{n-j+1}
$$
we end up with a differential inequality
$$
\frac{d { E}(\f_t)}{dt} \geq -C_1+C_2 E(\f_t),
$$
for some constants $C_1,C_2 \geq 0$. It follows that  $t \mapsto E(\f_t)+C t $ is increasing on $[0,T]$
for an appropriate choice of $C=C_T$.
\end{proof}

\subsubsection{The general case}

Recall that the integrability exponent of $\f_0$ at point $x \in X$ is
$$
 c (\f_0,x)  := \sup \{c > 0  \, | \, e^{- 2c \, \f_0} \in L^1 (V_x)\},
 $$ 
where $V_x$ denotes an arbitrarily small neighborhood of $x$. We let
$$
c(\f_0):= \sup \{ c (\f_0,x)  \, | \, x \in X\}
$$
denote the uniform integrability index of $\f_0$. It follows from Skoda's integrability theorem that
$$
\frac{1}{\nu(\f_0,x)} \leq c(\f_0,x) \leq \frac{n}{\nu(\f_0,x)},
$$
where $\nu(\f_0,x)$ denotes the Lelong number of $\f_0$ at $x$. Thus $c(\f_0)=+\infty$ if and only if $\f_0$ has zero lelong number at all points.

\smallskip

Fix $0<\b<c(\f_0)$ and $0<\a$ such that 
$$
\chi +(2\b-\a) \omega \geq 0.
$$
 The measure $e^{-2\b \f_0} \mu$ is absolutely continuous with density in 
$L^p$, for some $p>1$. It follows from Kolodziej's uniform estimate \cite{Kol98} that there exists a unique continuous $\omega$-psh funtion $u$ such that 
$$
\a^n (\omega+dd^c u)^n= e^{\a u -2\b \f_0} \mu.
$$

\begin{lem} \label{lem:minorationgenerale}
For $0<t< \min(T_{max},1/2\beta)$ and for all $x \in X$,
$$
(1-2\b t)\f_0(x)+\a tu(x)+n(t \log t-t) \leq \f_t(x).
$$
\end{lem}

\begin{proof}
Set $\p_t:=(1-2\b t)\f_0+\a tu+n(t \log t-t)$. It follows from our choice of $\a,\b$ that
$\p_t$ is $\theta_t$-psh since
$$
\theta_t+dd^c \p_t=(1-2\b t) \omega_{\f_0}+\a t \omega_u+t[\chi +(2\b-\a) \omega] \geq 0.
$$
Moreover $\p_t$ is a subsolution of (CMAF),
$$
(\theta_t+dd^c \p_t)^n \geq \a^n t^n \omega_u^n=e^{\dot{\p}_t} \mu,
$$
as the reader will easily check. Since $\p_0=\f_0$, the conclusion follows from the maximum principle.
\end{proof}

Note for  later use that this lower bound shows in particular that the integrability exponent (resp. the Lelong number) of $\f_t$ at point $x$ increases (resp. decreases) linearly in time.

\section{Bounds on $\dot{\f_t}$}

In this section again  we assume  that $\f_t$ satisfies $(CMAF)$ with an initial data $\f_0$, which
is a smooth strictly $\omega$-psh function.

\subsection{Bounding $\dot{\f_t}$ from above}

The following elementary estimate is, together with Kolodziej's a priori bound, a key to 
establish the smoothing property of complex Monge-Amp\`ere flows:

\begin{prop} \label{pro:clef}
There exists $C=C(\sup_X \f_0,-\inf_X h)>0$ such that for all $t>0$ and $x \in X$,
 $$
\dot{\f_t}(x) \leq \frac{-\f_0(x)+C}{t}+C
$$
\end{prop}

\begin{proof}
Consider $H(t,x):=t\dot{\f_t}(x)-(\f_t-\f_0)(x)-nt$.  Observe that 
$$
\frac{\partial H}{\partial t}=t\ddot{\f_t}-n=t \Delta_{\omega_t} \dot{\f_t}+t \tr_{\omega_t} \chi-n
$$
while
$$
\Delta_{\omega_t} H=t \Delta_{\omega_t} \dot{\f_t}-\Delta_{\omega_t} (\f_t-\f_0)
=t \Delta_{\omega_t} \dot{\f_t}-[n-\tr_{\omega_t}(T_0)-t \tr_{\omega_t}(\chi) ]
$$
therefore
$$
\left( \frac{\partial }{\partial t}-\Delta_{\omega_t} \right) H=-\tr_{\omega_t}(T_0) \leq 0.
$$

It thus follows from Proposition \ref{prop:max} that
$H$ attains its maximum on $(t=0)$. Now $H(0,x) \equiv 0$ hence Lemma \ref{lem:majosup} 
yields the conclusion.
\end{proof}

Note that similar bounds involving $\inf_X \f_0$ were previously obtained by Song-Tian \cite{ST}.
When $\dot{\f_0}$ is bounded from above, one can also show that
$\dot{\f_t}$ is uniformly bounded from above (see Szekelyhidi-Tosatti \cite{SzTo}
or Proposition \ref{prop:upperbounddensity}). However all these  estimates
require the initial data $\f_0$ to be bounded, an hypothesis which we want to avoid here.

\subsection{Bounding the oscillation of $\f_t$}

Observe that $\f_t$ is a family of $\theta_t$-psh functions such that
$$
(\theta_t+dd^c \f_t)^n=F_t \omega^n,
$$
where for each fixed $t>0$, the densities
$$
F_t=\exp(\dot{\f_t}+h) \leq \exp\left( \frac{-\f_0(x)+C}{t}+C' \right)
$$
are uniformly in $L^2(\omega^n)$ if $\f_0$ has zero Lelong number at all points (by Skoda's integrability theorem \cite{Sko}). It follows therefore from the uniform version of Kolodziej's estimates (see \cite{Kol98,EGZ08})
that the oscillation of $\f_t$ is uniformly bounded:

\begin{thm} \label{thm:oscillation}
Assume $\f_0$ has zero Lelong number at all points $x \in X$.
For each $t>0$, there exists $M(t)>0$ independent of $\inf_X \f_0$ such that
$$
\rm{Osc}_X(\f_t) \leq M(t).
$$

When $\f_0$ has some positive Lelong number and $T_{max}$ is large 
enough \footnote{e.g. when $\{\eta\}-c_1(X)=\{\chi\}$ is nef so that $T_{max}=+\infty$.}, 
then the Lelong numbers of $\f_0/t$ become so small (when $t$ is large) that 
$f_t$ is uniformly in $L^{1+\e}(\omega^n)$
for $t \geq t_{\e}$ and the same conclusion thus applies.
\end{thm}

\subsection{Bounding $\dot{\f_t}$ from below}

The following estimate is due to  Song-Tian \cite[Lemma 3.2]{ST}:

\begin{prop} \label{pro:STbelow}
Fix $0<T<T_{max}$ and
assume $\f_0$ is bounded. Then for all $(x,t) \in X \times ]0,T]$,
 $$
\dot{\f_t}(x) \geq n \log t-A Osc_X \f_0-C,
$$
where $A>1/(T_{max}-T)$ and $C=C(A)>0$.
\end{prop}

We include a proof for the reader's convenience.

\begin{proof}
Consider $G(t,x)=\dot{\f_t}(x)+A \f_t(x)-H(t)$, where $H:\R_*^+ \rightarrow \R$ is  a smooth function 
to be specified hereafter. We let the reader check that
$$
\left( \frac{\partial}{\partial t}-\Delta_t \right)(G)
=A \dot{\f_t}+\tr_{\omega_t}(A \theta_t+\chi)-H'(t)-An.
$$

Now $A\theta_t+\chi=A \theta_s \geq A\omega/2$ with $s=t+1/A<T_{max}$ since we assume
$A$ is so large that $1/A <T_{max}-T$ hence
$$
\tr_{\omega_t}(A \theta_t+\chi) \geq A \tr_{\omega_t}(\omega)/2 \geq \frac{f_t^{-1/n}}{C},
$$
where $\log f_t:=\dot{\f_t}$ and we have used the inequality
$$
\tr_{\omega_t}(\omega) \geq n \left( \frac{\omega_t^n}{\omega^n} \right)^{-1/n}
\geq {{f_t}^{-1/n}}{e^{-\sup_X h/n}}.
$$
Observe finally that $\e x > \log x-C_\e$ for all $x>0$ to conclude that
$$
\left( \frac{\partial}{\partial t}-\Delta_t \right)(G)
>  \frac{f_t^{-1/n}}{C_1}-H'(t)-C_2.
$$

The first condition we impose on $H$ is that $H(0)=-\infty$. This insures that the function $G$
attains its minimum on $ [0,T]\times X$ at a point $(x_0,t_0)$ with $t_0>0$. At this point we
therefore have a control on the density $f_t$, namely
$$
C_1[C_2+H'(t_0)] \geq f_{t_0}^{-1/n}(x_0) 
$$
which yields 
$$
G(t_0,x_0) \geq A \f_{t_0}(x_0)-C_3-\{ n \log[C_2+H'(t_0)]+H(t_0) \}
$$

It follows from Lemmata \ref{lem:majosup} and  \ref{lem:minoinf} that 
$\f_{t_0}(x_0) \geq \inf_X \f_0-C'$ and $\f_t(x) \leq \sup_X \f_0+C''$, thus
$$
\dot{\f_t} \geq H(t)-A Osc_X \f_0-C_4-\{ n \log[C_2+H'(t_0)]+H(t_0) \}
$$

The second condition we now would like to impose on $H$ is that
$$
H+n \log \left[ C_2+H' \right] \leq C'''
$$
is uniformly bounded from above on $]0,T]$. We let the reader check that the "best" function
satisfying both conditions and not going too fast to $-\infty$ as $t$ goes to zero is $H(t)=n \log t$.
For such a choice we obtain the desired inequality.
\end{proof}

Observe that if we start the complex Monge-Amp\`ere flow $(CMAF)$ from the initial data
$\f_s$, $s \geq 0$, we obtain (by uniqueness of the solution) $\f_{t+s}$. Thus
$$
\dot{\f}_{t+s}(x) \geq n \log t-A Osc_X \f_s-C.
$$
 From Theorem \ref{thm:oscillation} we infer the following important consequence:

\begin{cor} \label{cor:minodot}
Assume $\f_0$ has zero Lelong number at all points and fix $T<T_{max}$. Then there exists
$\kappa: ]0,T] \rightarrow \R^+$ a decreasing function  such that for all $0<t \leq T$ and $x \in X$,
$$
\dot{\f}_t(x) \geq -\kappa(t).
$$

When $\f_0$ has positive Lelong number at some points, the conclusion merely applies
when $t_\e \leq t \leq T$.
\end{cor}

\subsection{Further bounds on the densities}

Here we consider  the  case when the initial data $\f_0$ is such that
its Monge-Amp\`ere measure $(\omega+dd^c \f_0)^n=f_0 e^h \omega^n$ is absolutely continuous
with respect to Lebesgue measure. 

An abstract measure theoretic argument shows that $f_0$ actually belongs to an Orlicz class: there exists a convex increasing function $w:\R^+ \rightarrow \R^+$ such that
$w(t)/t \rightarrow +\infty$ as $t \rightarrow +\infty$ and
$$
\int_X w \circ f_0 \, d\mu <+\infty.
$$

\begin{prop} \label{prop:upperbounddensity}
Assume that $\chi \leq 0$. Then
the map $t\mapsto \int_X w \circ f_t \, d\mu$ is decreasing along the complex Monge-Amp\`ere flow, where $f_t:=(\theta_t+dd^c \f_t)^n / \mu$.

In particular if $f_0 \in L^{\infty}(X)$ is bounded from above, then so is $f_t$, with
$$
\dot{\f_t} =\log f_t \leq  \log \sup_X f_0.
$$
\end{prop}

\begin{proof}
Set $I(t):=\int_X w \circ f_t \, d\mu$. Using that $\dot{\f_t}=\log f_t$, we observe that
$\log f_t$ satisfies the following Heat-type equation,
$$
\frac{\partial f_t}{\partial t}=\Delta_{\omega_t} f_t-\frac{|\nabla_{\omega_t} f_t|^2}{f_t}+
f_t \, \tr_{\omega_t}(\chi).
$$
Assuming $\chi \leq 0$ we infer
$$
I'(t) \leq n\int_X \frac{w' \circ f_t}{f_t} \, dd^c f_t \wedge \omega_t^{n-1}
-n\int_X \frac{w' \circ f_t}{f_t^2} \, df_t \wedge d^c f_t \wedge \omega_t^{n-1}.
$$
Integrating by parts yields
\begin{eqnarray*}
\lefteqn{  \! \! \! \! \! \! \!  \! \! \! \! \! \! \! 
\int_X \frac{w' \circ f_t}{f_t} \, dd^c f_t \wedge \omega_t^{n-1}
= -\int d \left( \frac{w' \circ f_t}{f_t} \right) \wedge d^c f_t \wedge \omega_t^{n-1}} \\
&=& -\int_X \frac{w'' \circ f_t}{f_t} \, df_t \wedge d^c f_t \wedge \omega_t^{n-1}
+\int_X \frac{w' \circ f_t}{f_t^2} \, df_t \wedge d^c f_t \wedge \omega_t^{n-1},
\end{eqnarray*}
therefore
$$
I'(t) \leq -n \int_X \frac{w'' \circ f_t}{f_t} \, df_t \wedge d^c f_t \wedge \omega_t^{n-1} \leq 0.
$$
as claimed.

Assume now that $f_0 \in L^{\infty}(X)$. In particular $f_0 \in L^p(X)$ for all $p>1$ and we've just seen
that  $\|f_t\|_{L^p} \leq \|f_0\|_{L^p}$. Letting $p \rightarrow +\infty$ yields the desired control.
\end{proof}

This computation is  a generalization of an observation due to Chen-Tian-Zhang 
\cite[Lemma 2.4]{CTZ11},
who consider the case when $f_0 \in L^p$ for some $p>1$. In this case $\f_0$ is continuous,
as follows from Kolodziej's  estimate \cite{Kol98}.  

\smallskip

For stronger lower bounds, we note the following immediate consequence of the minimum principle (Proposition \ref{prop:max}):

\begin{prop}
Assume $\chi \geq 0$.
If $f_0 >0$ is uniformly bounded away from zero, then so is $f_t$ 
for all $t>0$ with
$$
\dot{\f_t}=\log f_t \geq \log \inf_X f_0.
$$
\end{prop}

 Even though it requires a strong condition on the initial data, such a lower bound could be useful in some cases (see e.g. \cite{SzTo}).

\section{Higher order estimates}

We assume here again that $\f_t$ satisfies $(CMAF)$ with an initial data $\f_0$, which
is a smooth strictly $\omega$-psh function.

\subsection{Preliminary results}

We shall need two standard auxiliary results (see \cite{Yau, Siu} for a proof):

\begin{lem} \label{lem:comparaisontraces}
Let $\a,\b$ be positive $(1,1)$-forms. Then
$$
 n \left(\frac{\a^n}{\b^n}\right)^{\frac{1}{ n}} \leq Tr_{\b}(\a)
\leq n \left( \frac{\a^n}{\b^n} \right) \cdot \left( Tr_{\a} (\b) \right)^{n-1}.
$$
\end{lem}

Applying these inequalities to $\a=\omega_t:=\theta_t+dd^c \f_t$ and $\b=\omega$,
we obtain:

\begin{cor} \label{cor:comparaisontraces}
For all $0<t \leq T <T_{max}$, 
there exists $C(t)>0$ which only depends on $\|\dot{\f_t}||_{L^{\infty}}, \|h||_{L^{\infty}}$ such that
$$
\frac{1}{C(t)} \leq Tr_{\omega}(\omega_t) \leq C(t) [Tr_{\omega_t}(\omega)]^{n-1}.
$$
\end{cor}

\begin{lem} \label{lem:Siu}
Let $\omega,\omega'$ be arbitrary K\"ahler forms.
Let $-B \in \R$ be a lower bound on the holomorphic bisectional curvature of $(X,\omega)$. Then
$$
\Delta_{\omega'} \log Tr_{\omega} (\omega') \geq -\frac{Tr_{\omega}(Ric(\omega'))}{Tr_{\omega}(\omega')}
-B \,  Tr_{\omega'}(\omega).
$$
\end{lem}

\subsection{Bounding $\Delta \f_t$ from above}

\begin{lem} \label{pro:c2baby}
Fix $0<T<T_{max}$. Then for all $x\in X$ and $s,t>0$ such that $s+t \leq T$,
$$
0 \leq t \log Tr_{\omega} (\omega_{t+s}) \leq A Osc_X(\f_s) +C+[C-n \log s +A Osc_X(\f_s)]t
$$
for some uniform constants $A,C>0$.
\end{lem}

Applying this inequality for $t/2$ and $s=t/2$, we obtain:

\begin{cor}
Fix $0<T<T_{max}$. Then for all $(t,x)\in [0,T] \times X$
$$
0 \leq t \log Tr_{\omega} (\omega_{t}) \leq 2A Osc_X(\f_{t/2}) +C'.
$$
\end{cor}

\begin{proof}
Set $\a:=t \log u-A \f_{t+s}$, where $u:=Tr_{\omega} (\omega_{t+s}) $ and $A>0$ will be specified later.
The desired inequality will follow if we can bound $\a$ from above. Set $\Delta_{t}=\Delta_{\omega_{t+s}}$ and observe that
$$
\left(\frac{\partial}{\partial t}-\Delta_t\right)(\a)
=\log u + \frac{t}{u} \frac{\partial u}{\partial t}-A \dot{\f}_{t+s}
-t \Delta_t \log u+A \Delta_t \f_{t+s}.
$$

The last term yields 
$A \Delta_t \f_{t+s}=An -A Tr_{\omega_{t+s}}(\theta_{t+s}) \leq An -\frac{A}{2} Tr_{\omega_{t+s}}(\omega)$. 
The last but one is estimated thanks to Lemma \ref{lem:Siu},
$$
-t \Delta_t \log u \leq Bt \, Tr_{\omega_{t+s}}(\omega)+
t \frac{Tr_{\omega}(Ric (\omega_{t+s}))}{Tr_{\omega}(\omega_{t+s})}.
$$
Since
$$
\frac{t}{u} \frac{\partial u}{\partial t}=\frac{t}{u} \Delta_t (\log \omega_{t+s}^n/\omega^n)
=\frac{t}{u} \left\{ -Tr_{\omega}(Ric \, \omega_{t+s}) +Tr_{\omega}(\chi+Ric \, \omega-dd^c h) \right\},
$$
we infer 
$$
-t \Delta_t \log u+\frac{t}{u} \frac{\partial u}{\partial t} \leq 
 (B+C_1) t \, Tr_{\omega_{t+s}}(\omega)+C_2,
$$
using that $Tr_{\omega}(\chi+Ric \, \omega-dd^c h)$ is uniformly bounded below
and the elementary inequality $n \leq  Tr_{\omega_{t+s}}(\omega)  Tr_{\omega}(\omega_{t+s})$.

It follows now from Lemma \ref{lem:comparaisontraces} and the inequality 
$(n-1)\log x <x+C_n$ that
$$
\log u \leq \dot{\f}_{t+s} +C_3+ Tr_{\omega_{t+s}}(\omega)
$$
Altogether this yields
$$
\left(\frac{\partial}{\partial t}-\Delta_t\right)(\a)
\leq C_4- (A-1) \dot{\f}_{t+s} +\left[ (B+C_1) t +1-A/2 \right] Tr_{\omega_{t+s}}(\omega).
$$

We choose $A>0$ so large that $(B+C_1)t +1-A/2 <0$. The desired inequality now follows from the maximum principle and Lemma \ref{lem:majosup}, together with Proposition \ref{pro:STbelow}.
\end{proof}

\subsection{Evans-Krylov and Schauder}

Using the complex parabolic Evans-Krylov theory together with Schauder's estimates
(see \cite[Theorem 3.1.4]{BG13}), 
it follows from  previous results that the following higher order a priori estimates hold:

\begin{thm}  \label{thm:higherorder}
Fix $0<T<T_{max}$.
If $\f_0$ has zero Lelong number at all points $x \in X$ then for each $\e>0$ and $k \in \N$, there exists $C_k(\e) >0$ such that
$$
\| \f_t\|_{{\mathcal C}^k(X \times [\e,T])} \leq C_k(\e).
$$

If $\f_0$ has positive Lelong number at some points, the same results hold but only on some time interval
$[t_\e, T]$.
\end{thm}

 We will analyze more precisely the value of $t_\e$ in section \ref{sec:concluding}.

\section{Proof of the main theorem}

We are now ready to prove Theorem \ref{thm:main}.

\subsection{Defining the flow}

Let $\f_0$ be an $\omega$-psh function. 
Using \cite{Dem92} (or \cite{BK07}), we approximate $\f_0$ by a decreasing sequence $\f_{0,j}$
of smooth $\omega$-psh functions. We let $\f_{t,j}$ denote the smooth family of 
$\theta_t$-psh functions, solution of the parabolic flow
$$
\hskip-4cm
(CMAF) \hskip3cm \frac{\partial \f_{t,j}}{\partial t}=\log \left[ \frac{(\theta_t+dd^c \f_{t,j})^n}{\omega^n} \right]
$$
with initial value $\f_{0,j}$, which is well defined on $[0,T_{max}[ \times X$.

It follows from Corollary \ref{cor:envelope} that for $(t,x)$ fixed, $j \mapsto \f_{t,j}(x)$ is non-increasing.
We can thus set
$$
\f_t(x):=\lim_{j\rightarrow +\infty}\searrow \f_{t,j}(x).
$$

It follows from Lemma \ref{lem:minorationgenerale} that the functions $t \mapsto \sup_X \f_{t,j}$
are uniformly bounded, thus $\f_t$ is not identically $-\infty$
hence it is a well defined $\theta_t$-psh function.

It moreover follows from Theorem \ref{thm:higherorder}
that for all $t>0$, $\f_t$ is smooth and satisfies 
$$
\frac{\partial \f_t}{\partial t}=\log \left[ \frac{(\theta_t+dd^c \f_t)^n}{\omega^n} \right]
$$
when $\f_0$ has zero Lelong number at all points.

Note that $(\f_t)$  is relatively compact in $L^1$ as $t \rightarrow 0^+$.
We show in the next section that $\f_t \longrightarrow \f_0$ as $t \searrow  0^+$, where the stronger the regularity assumption on $\f_0$, the stronger the convergence.

\subsection{Continuity at zero}

Let $\p=\lim \f_{t_k}$ 
be a cluster value of $(\f_t)$, as $t \rightarrow 0^+$.
It is a standard property of quasi-plurisubharmonic functions that for all $x \in X$,
$$
\p(x) \leq \limsup \f_{t_k}(x),
$$
with equality almost everywhere.
We claim that 
$$
\p \leq \f_0.
$$
   It follows indeed from Corollary \ref{cor:envelope} that 
$\f_{t,j}$ decreases to $\f_t$, hence for all $j \in \N$,
$$
\p(x) \leq \limsup \f_{t_k}(x) \leq  \limsup \f_{t_k,j}(x)=\f_{0,j}(x)
$$
since $ \f_{t,j}$ is continuous at $t=0$.
The conclusion follows by letting $j \rightarrow +\infty$.

We now analyze various settings in which we establish the reverse inequality.

\subsubsection{Convergence in capacity} 

When the initial data $\f_0$ is continuous, it follows from Corollary \ref{cor:envelope}
that $\f_t \in {\mathcal C}^0([0,T_{max}[ \times X)$, hence $\f_t$ uniformly converges towards
$\f_0$ as $t \rightarrow 0$.  

We can not expect uniform convergence when $\f_0$ is merely bounded; it follows however from 
Lemma \ref{lem:minoinf} that 
$\p \geq \f_0$.  Thus 
$$
\f_t \rightarrow \f_0
$$
in this case and the convergence moreover holds {in capacity: this roughly says that the convergence is uniform outside sets of arbitrarily small capacity. It is the strongest convergence one can expect in this bounded context
(see \cite{GZ05}).

\begin{rem}
Observe that for all $x \in X$,
$$
(\f_t-\f_0)(x)=\int_0^t \dot{\f_s}(x) ds=\int_0^t \log f_s(x)ds,
$$
where $f_s=(\theta_s+dd^c \f_s)^n / \mu$.
We have shown  earlier that, when $\f_0$ is bounded,  
$$
f_s \geq s^n/C  \; \; \; \text{ while }  \; \; \;  \log f_s \leq C/s, 
$$
for some appropriate constant $C>0$. Although the upper bound is much weaker, it is not far from being optimal if the initial data is bounded but not continuous: indeed if we could uniformly bound from above 
$\log f_s(x)$ by $h(s)$ with $h$ integrable at zero,
it would follow that $\f_t$ uniformly converges towards $\f_0$ as $t \rightarrow 0$, hence $\f_0$ would  
have to be continuous.
\end{rem}

\subsubsection{Convergence in energy} 

We consider here the case when  $\f_0$ has finite energy,
$$
E(\f_0):=\frac{1}{(n+1)V} \sum_{j=0}^n \int_X \f_0 (\omega+dd^c \f_0)^j \wedge \omega^{n-j} >-\infty.
$$
We refer the interested reader to \cite{GZ07,BEGZ} for various information on the 
finite energy class ${\mathcal E}^1(X,\omega)$ of those $\omega$-psh functions
with finite energy.

\begin{prop}
Assume  $\f_0 \in {\mathcal E}^1(X,\omega)$.
As $t$ decreases to zero, the functions $\f_t$ converge to $\f_0$ in energy
(hence in particular in $L^1(X)$).
\end{prop}

\begin{proof}
It follows from Lemma \ref{lem:energymonotone} that the functions $\f_t$ stay in a compact subset of 
${\mathcal E}^1(X,\omega)$. Let $\p \leq \f_0$ be a cluster point as $t \rightarrow 0$.

Recall from \cite{BEGZ} that the energy $E(\cdot)$ is upper semi-continuous (for the weak
$L^1$-topology), it follows therefore from Lemma \ref{lem:energymonotone} that
$$
E(\f_0) \leq \lim_{t_j \rightarrow 0} E(\f_{t_j}) \leq E(\p) \leq E(\f_0),
$$
where the latter inequality follows from the fact that $\f \mapsto E(\f)$ is monotone increasing.
We infer $E(\p)=E(\f_0)$, whence $\f_0=\p$, as desired.

This not only shows that $\f_t \rightarrow \f_0$ in the $L^1$-sense as $t \rightarrow 0$, but also
that it does so in energy, in the sense of \cite{BBGZ,BBEGZ}. 
\end{proof}

Let us stress an important feature of the convergence in energy: it guarantees the continuity of the complex Monge-Amp\`ere measures 
$$
(\theta_t+dd^c \f_t)^n \stackrel{t \rightarrow 0}{\longrightarrow} (\omega+dd^c \f_0)^n,
$$
whereas these  operators are usually discontinuous for the  weaker $L^1$-convergence of potentials,
which we now consider.

\subsubsection{$L^1$-convergence}

We finally treat the general case of an arbitrary initial data $\f_0 \in L^1$.
Recall from  Lemma \ref{lem:minorationgenerale} that 
$$
(1-\b t) \f_0 + \a t u +n (t \log t-t) \leq \f_t
$$
where $u$ is a continuous $\omega$-psh function, thus
$$
\f_0 \leq \liminf_{t \rightarrow 0} \f_t.
$$
Since $\p \leq \f_0$, we infer $\p \equiv \f_0$, as desired.

\subsection{Uniqueness}

Let $(\f_t)_{t>0}$ be the solution to $(CMAF)$ constructed above by approximation, and
assume $(\p_t)_{t>0}$ is another family of smooth $\theta_t$-psh functions
which satisfy, on $]0,T_{max}[ \times X$,
$$
 \frac{\partial \p_t}{\partial t}=\log \left[ \frac{(\theta_t+dd^c \p_t)^n}{\mu} \right],
$$
with $\p_t \rightarrow \f_0$ as $t \searrow 0^+$.

We claim that $\p_t(x) \leq \f_t(x)$ for all $t,x$. Let $\f_{0,j}$
be a family of smooth $\omega$-psh functions decreasing to $\f_0$.
By construction the $\theta_t$-psh functions $\f_{t,j}$ are smooth on $[0,T_{max}[ \times X$
and decrease pointwise to $\f_t$. It thus suffices to show that $\p_t \leq \f_{t,j}$, for all fixed $j \in \N$.

Fix $\e>0$. It follows from the maximum principle that the smooth function $\p_t-\f_{t,j}$ attains 
its maximum on $[\e,T] \times X$ at a point $(\e,x_{\e})$, thus
$$
\p_t(x)-\f_{t,j}(x) \leq \sup_X (\p_{\e}-\f_{\e,j}).
$$
Since $(\e,x) \mapsto \f_{\e,j}(x)$ is continuous, it follows from Hartogs lemma that
$$
\sup_X (\p_{\e}-\f_{\e,j}) \stackrel{\e \rightarrow 0}{\longrightarrow} \sup_X (\f_0-\f_{0,j}) \leq 0.
$$
Thus $\p_t(x) \leq \f_{t,j}(x)$ for all $(t,x,j)$ and the conclusion follows.

\smallskip

The uniqueness property we have just established is called "maximally stretched" by P.Topping in complex dimension $n=1$ (see \cite{Top10, GT11}).

\section{Currents with positive Lelong numbers} \label{sec:concluding}

\subsection{Long time existence}

As we have noted on several occasions, our construction of the (twisted) K\"ahler-Ricci
flow also yields interesting information in case the initial current $T_0$ has positive Lelong numbers.

Recall that the (global) integrability index of $T_0=\omega+dd^c \f_0$ is 
$$
c(T_0):=\sup \{ c >0 \, | \, \int_X e^{-2c \f_0} \omega^n <+\infty \}.
$$
The definition clearly does not depend on the choice of potential.

\smallskip

When $T_0$ has positive Lelong numbers, it follows from Skoda's integrability theorem \cite{Sko} that
$c(T_0)$ is positive and finite. The "openess conjecture" of Demailly-Kollar \cite{DK01}) 
asserts that $e^{-2c(T_0)\f_0} \notin L^1$.
It has been solved very recently by  Berndtsson \cite{Bern13}
(the two dimensional case was previously settled in \cite{FJ05}).
 When the cohomology class of $\omega$ is rational and $T_0=[D]$ is
the current of integration along a 
(rational) effective divisor, the exponent $c(T_0)=lct(D)$ is the log-canonical threshold of 
$D$, an important algebraic invariant.

\medskip

If $T_{max} >1/c(T_0)$ it follows from Proposition \ref{pro:clef} that we have good control on
the time derivative $\dot{\f}_t$ for $1/2c(T_0) < t < T_{max}$, namely
$$
|| e^{\dot{\f_t}}||_{L^{1+\e}} \leq C_\e
$$
for all $\e>0$ and $(1+\e)/2c(T_0) \leq  t \leq T_\e< T_{max}$. The following result
 therefore follows from 
Kolodziej's estimates and our previous analysis:

\begin{thm} \label{thm:regLelong}
Let $T_0 \in \a_0$ be a positive current such that $1/2c(T_0)<T_{max}$. There exists a unique 
maximal family $(\omega_t)_{0<t<T_{max}}$ of positive currents, whose potentials are 
the decreasing limits of parabolic flows approximating the potential of $T_0$;
$\omega_t$ weakly converges towards $T_0$, as $t \searrow 0^+$.

When $t> 1/2c(T_0)$ these are smooth K\"ahler forms such that
$$
\frac{\partial \omega_t}{\partial t}=-Ric(\omega_t)+\eta.
$$
\end{thm}

This basically means that the singularities of $T_0$ are slowly attenuated 
(by Lemma \ref{lem:minorationgenerale} the Lelong numbers of $\f_t$ are linearly decreasing)
so that $\omega_t$ eventually becomes a smooth K\"ahler form.
This should be compared with Demailly's use of Kiselman's technique of attenuation of singularities:
when regularizing a positive current in \cite{Dem92}, the presence of positive Lelong numbers is an obstruction to the approximation with small loss of positivity.

\smallskip

When $c_1(X)-\{ \eta\} \leq 0$ so that $T_{max}=+\infty$, we therefore obtain the following generalization of Cao's celebrated result:

\begin{thm} \label{thm:cao} \cite{Cao85,ST12}
If $c_1(X)-\{ \eta\} \leq 0$, the twisted K\"ahler-Ricci flow continuously deforms {\it any} positive closed $(1,1)$-current to a canonical metric.
\end{thm}

On the other hand there might be no smoothing at all when $T_{max}<1/2c(T_0)$. 
We analyze more precisely this problem in the next section.

\subsection{The normalized K\"ahler-Ricci flow on Fano manifolds}

In this section we have a closer look at the smoothing properties of the
normalized K\"ahler-Ricci flow on a Fano manifold, i.e. a manifold with positive first Chern class
$c_1(X)$. 

We assume here $\a_0=c_1(X)$ and the Ricci-flow is normalized so as to have constant volume,
$$
\frac{\partial \omega_t}{\partial t}=-\Ric(\omega_t)+\omega_t.
$$
One classically passes from the (unnormalized) K\"ahler-Ricci flow to the normalized one by
rescaling in space and time: if $\Omega_s$ solves the K\"ahler-Ricci flow, then
$$
\omega_t=e^t \, \Omega_{1-e^{-t}}
$$
solves the normalized K\"ahler-Ricci flow. At the level of potentials we obtain
$$
\hskip-3.5cm(NCMAF) \hskip2cm  \frac{\partial \f_t}{\partial t}=\log \left[ \frac{(\omega+dd^c \f_t)^n}{e^h \, \omega^n} \right]+\f_t,
$$
where $\omega$ is a fixed K\"ahler form in $c_1(X)=\a_0$, $\omega_t=\omega+dd^c \f_t$ and
$$
\Ric(\omega)=\omega+dd^c h.
$$

Our analysis thus applies here again, showing that we can run this flow from an arbitray 
positive current $T_0=\omega+dd^c \f_0$ with zero Lelong numbers. The flow exists here for all times
$t>0$, i.e. the "normalized" $T_{max}$ is $+\infty$.

We  can adapt the proof of Proposition \ref{pro:clef} and obtain
here:

\begin{prop} \label{pro:majoNKRF}
For all $x \in X$ and $t>0$,
$$
\dot{\f_t}(x) \leq \frac{-\f_0(x)+e^t \sup_X \f_0+nt}{1-e^{-t}}
$$
\end{prop}

\begin{proof}
Consider $H(t,x)=(1-e^{-t}) \dot{\f_t}-(\f_t-\f_0)-nt$ and show that
$$
\left( \frac{\partial}{\partial t}-\Delta_t \right)H=-\tr_{\omega_t}(T_0) \leq 0.
$$

The upper bound on $\f_t$ can be obtained by applying the maximum principle to 
the function $e^{-t} \f_t(x)$.
\end{proof}

Although the flow exists for all times $t>0$, the attenuation of singularities is limited,
as $1-e^{-t} \nearrow 1$ as $t \nearrow +\infty$ (in contrast to $t \nearrow +\infty$ in
Proposition \ref{pro:clef}).
One can indeed not always expect smoothing properties in this case as the following simple example shows:

\begin{exa} \label{exa:pn}
Let $X=\C\PP^n$ equipped with a K\"ahler-Einstein Fubini-Study metric 
$\omega=(n+1) \omega_{FS} \in c_1(\C\PP^n)$, 
whose potential is written in some homogeneous coordinates as
$$
\rho=\frac{(n+1)}{2}\log\sum_{i=0}^n |z_i|^2.
$$
We can rescale these coordinates and still obtain a K\"ahler-Einstein metric (image of the initial one by a holomorphic automorphism), for example
$$
\f^j=\frac{(n+1)}{2}\log \left[|z_0|^2/j+\sum_{i=1}^n |z_i|^2 \right]-(n+1)\log ||z||
$$
is such that $\omega^j:=\omega+dd^c \f^j$ is K\"ahler-Einstein. Observe that
$$
\f_0=(n+1) \log||(z_1,\ldots,z_n)||-(n+1) \log ||z||=\lim_{j \rightarrow +\infty} \searrow\f^j
$$
is such that $c(\f_0)=n/(n+1)$.

Since the $\omega^j$ are fixed points of the normalized K\"ahler-Ricci flow, our approximation process
yields in this case a stationary family $\f_t \equiv \f_0$, hence there is no regularization at all here !
A similar construction can be made on any K\"ahler-Einstein Fano manifold 
admitting non trivial holomorphic vector fields.
\end{exa}

On the other hand we have  the following positive result, which allows to extend Perelman's celebrated convergence result:

\begin{thm} \label{thm:fano}
Let $X$ be a Fano manifold.
Let $T_0 \in c_1(X)$ be a positive current with integrability index $c(T_0)>1$. 
Then the normalized
K\"ahler-Ricci flow can be run from $T_0$ and yields a family of positive currents 
$\omega_t$ which are smooth K\"ahler forms for $t \geq 1$.

If $D_0$ is a smooth anticanonical divisor, then one can also run the normalized
K\"ahler-Ricci flow with initial data $T_0=[D_0]$.

In particular if $X$ admits a K\"ahler-Einstein, then the normalized
K\"ahler-Ricci flow continously deforms any such $T_0$ to a K\"ahler-Einstein metric.
\end{thm}

A similar result can be obtained for the twisted K\"ahler-Ricci flow, using the recent work of 
Collins-Szekelyhidi \cite{CSz12}.

\begin{proof}
Assume first that $T_0=\omega+dd^c \f_0$ with $c(\f_0)>1$.
Let $u$ be a continuous $\omega$-psh function such that 
$$
(\omega+dd^c u)^n =e^{2u-2 \f_0} \omega^n.
$$
That such a function exists follows from our assumption on the integrability index, together with Kolodziej's uniform estimate \cite{Kol98}. Observe that
$$
u_t=e^{-t} \f_0+(1-e^{-t}) u-Ct+n(t \, log t -t)
$$
is a family of $\omega$-psh functions which satisfy
$$
(\omega+dd^c u_t)^n \geq (1-e^{-t})^n (\omega+dd^c u)^n \geq \frac{ \min(t^n,1)}{2^n} 
e^{2u-2 \f_0} \omega^n,
$$
while
$$
e^{\dot{u}_t-u_t+h} \omega^n =t^n e^{(2u-2\f_0)e^{-t}-u+h-C} \leq (\omega+dd^c u_t)^n
$$
on $X \times [0,T]$, if $C>>1$ is large enough. 

It follows from the maximum principle and Proposition \ref{pro:majoNKRF} that $\f_t \geq u_t$ and
$$
e^{\dot{\f}_t-\f_t} \leq \exp \left( -\left[e^{-t}+\frac{1}{1-e^{-t}} \right]\f_0+C' \right) \in L^p,
$$
for some $p>1$ as soon as $t >-\ln \left(\frac{3-\sqrt{5}}{2} \right)$.

We now consider the case when $T_0=\omega+dd^c \f_0=[D_0]$ is the current of integration along a 
smooth anticanonical divisor. Note that $\f_0=\log |s_0|_h$, where
$s_0$ is a holomorphic section of $-K_X$ such that $D_0=(s_0=0)$ and $h$ is a smooth hermitian
metric with positive curvature form $\Theta_h=\omega$. In this case $c(T_0)=1$ but the Poincar\'e type
metric
$$
\mu:=e^{-2\f_0} \frac{\omega^n}{(\log|s_0|^2_h)^2}=\frac{\omega^n}{|s_0|^2_h (\log|s_0|^2_h)^2}
$$
has finite mass if $D_0$ is smooth or  not too singular (this is precisely what we need here). It follows from the work of Tian-Yau (see \cite{Auvray} for more recent information on this topic) that one can solve
the complex Monge-Amp\`ere equation
$$
(\omega+dd^c v)^n=c \, \mu,
$$
where $u$ is $\omega$-psh on $X$, smooth in $X \setminus D$, with mild singularities along $D$
(modelled on the potential $-\log (-\log |s_0|^2_h)$). In particular $v$ has zero Lelong number at all points
of $X$. This latter mild information can also be obtained as a particular case of the results developed in \cite{GZ07}. Observe that $u=v+C$ is then a solution of the complex Monge-Amp\`ere equation
$$
(\omega+dd^c u)^n =e^{2u-2 \f_0} \omega^n,
$$
for some appropriate choice of the normalizing constant $C \in \R$. We can from there proceed as in the case when $c(T_0)>1$.

\smallskip

The asymptotic behavior of $\omega_t$ as $t \rightarrow +\infty$, is the same as that of the 
normalized K\"ahler-Ricci flow with smooth initial value $\omega_1$, as follows from the semi-group property. It thus follows from a celebrated result of Perelman (see \cite{TZ07}) that 
when $t \rightarrow +\infty$,
$\omega_t$ converges to a K\"ahler-Einstein metric if there is one.
\end{proof}

\begin{rem}
Note that one can similarly continuously deform any positive current $T_0=\omega+dd^c \f_0$
with $c(\f_0)=1$, as long as $e^{-2\f_0}/[1+|\f_0|^N] \in L^1$. While it is now known that
$e^{-2\f_0}$ is not integrable at this critical exponent \cite{Bern13}, it is tempting to make the 
{\it closedness conjecture} that one can always reach integrability at the critical exponent by
adding by a polynomial factor.
\end{rem}

\begin{rem} \label{rem:heuristic}
Here follows a heuristic argument showing how the smoothing property of the K\"ahler-Ricci flow should help in analyzing the long term behavior of the normalized K\"ahler-Ricci flow on Fano varieties. 

Assume for simplicity that $X$ is a Fano manifold with no holomorphic vector field and pick
$\omega \in c_1(X)$ a K\"ahler form. Its Ricci curvature form decomposes as
$$
\Ric(\omega)=\omega+dd^c h,
$$
where we normalize the function $h$ (the Ricci deviation) so that $\int_X e^{h} \omega^n=1$.

Any other K\"ahler form in $c_1(X)$ writes
$\omega_\f=\omega+dd^c \f$, for some smooth $\omega$-psh function $\f$.
It follows from the work of Tian \cite{Tian97}  that $X$ admits a (unique) K\"ahler-Einstein 
$\omega_{KE}$ if and only if the functional
$$
{\mathcal F}: \omega_\f \mapsto E(\f)+\log \int_X e^{-\f+h} \omega^n
$$
is {\it proper} (it is then actually even {\it coercive} as was shown by Phong-Song-Sturm-Weinkove in
\cite{PSSW08}). Moreover $\omega_{KE}$ is then the unique maximizer of ${\mathcal F}$.
Here $E$ denotes the functional introduced earlier, namely
$$
E(\f)=\frac{1}{(n+1)V} \sum_{j=0}^n \int_X \f \, \omega_\f^j \wedge \omega^{n-j}.
$$

We now consider the normalized K\"ahler-Ricci flow $(\omega_t)_{t>0}$ starting from $\omega_0=\omega$ and the goal is to propose a scheme for the proof of Perelman's result that $\omega_t \rightarrow \omega_{KE}$ as $t \rightarrow +\infty$, assuming ${\mathcal F}$ is proper:
\begin{enumerate}
\item observe that ${\mathcal F}$ is non decreasing along the normalized K\"ahler-Ricci flow and even increasing unless we have reached $\omega_{KE}$ (entropy argument);
\item The properness assumption implies that the $\omega_t$'s stay in a  compact subset of the finite energy class ${\mathcal E}^1(c_1(X))$;
\item the functional ${\mathcal F}$ is contant on the set ${\mathcal C}$ of cluster values which consists of
finite energy currents;
\item one can run the normalized K\"ahler-Ricci flow starting from any $T_0 \in {\mathcal C}$, since 
${\mathcal F}$ is constant on ${\mathcal C}$, it follows from (1) that ${\mathcal C}=\{\omega_{KE}\}$.
\end{enumerate}
All these arguments however need delicate justifications, we refer the interested reader to \cite{BBEGZ} for 
more details.
\end{rem}

\subsection{The 2D-Ricci flow}

The smoothing properties of the Ricci flow in (real) dimension two has been intensively studied in the past twenty years, in connection with equations of porous medium type, notably
the logarithmic (fast) diffusion equation. There is no existence  nor uniqueness of solutions in general, if the initial measure has atoms (see \cite{Vaz06}).

If $X=\C\PP^1$ is the Riemann sphere and $\omega$ a Fubini-Study metric, we can write
$\omega_t=dd^c \p_t$ in a chart $\C \subset \C\PP^1$, with
$$
\frac{\partial \p_t}{\partial t}=\log \Delta_\omega \p_t
\; \; \text{ in } \C =\R^2.
$$
Setting $f_t=\Delta_\omega \p_t$, this transforms into
$$
\frac{\partial f_t}{\partial t} =\Delta_\omega \log f_t 
\; \; \text{ in } \; \; \C \times ]0,T]
$$
with initial data a positive Radon measure $T_0=dd^c \p_0 \geq 0$.
We refer the interested reader to \cite{GT11, R12,Top10, Top12} for recent works in this direction. It seems that our  results go beyond what was previously known  in this two-dimensional setting.

\section{The K\"ahler-Ricci flow on lt varieties} \label{sec:last}

In this final section we briefly explain how to generalize our previous results in two different directions.
We first show that one can even run the twisted K\"ahler-Ricci flow from a current representing a class that is merely nef. We then extend our previous analysis to the case of midly singular varieties.

\subsection{Starting from a nef class}

Our purpose here is to show that the twisted K\"ahler-Ricci flow can also be used to 
smooth out a positive current with zero Lelong numbers belonging to a nef class, as long
as the corresponding deformation at the level of cohomology immediately enters the K\"ahler cone.

\begin{thm} \label{thm:nef}
Assume that $\a_0$ is nef while $\{\eta\}-c_1(X)$ is a K\"ahler class.
Let $T_0 \in \a_0$ be a positive current with zero Lelong numbers. There exists a unique 
maximal family
$(\omega_t)_{0<t<T_{max}}$ of K\"ahler forms such that
$$
\frac{\partial \omega_t}{\partial t}=-Ric(\omega_t)+\eta,
$$
and $\omega_t$ weakly converges towards $T_0$, as $t \searrow 0^+$.
\end{thm}

The proof follows exactly the same lines as that of our previous results, so we only emphasize the main differences.

\begin{proof}
The problem can again be  rewritten at the level of potentials. Fix $\theta$ a smooth differential closed
 form representing $\a_0$ (note that $\theta$ is not necessarily semi-positive)
and let $\f_0 \in PSH(X,\theta)$ be a $\theta$-psh potential for $T_0$, i.e.
$T_0=\theta+dd^c \f_0$.
Fix $\omega \in \{\eta\}-c_1(X)$ a K\"ahler form and write $\omega_t=\theta+t \omega+dd^c \f_t$ so
that 
$$
\frac{\partial \omega_t}{\partial t}=-Ric(\omega_t)+\eta,
$$
is equivalent to
$$
 \frac{\partial \f_t}{\partial t}=\log \left[ \frac{(\theta+t \omega+dd^c \f_t)^n}{e^h \omega^n} \right]
$$
for some smooth real valued function $h$.

We already now how to solve the corresponding problem when 
$\a_0$ is a K\"ahler class, so it is natural to
approximate by replacing $\a_0$ with $\a_0+\e \{\omega\}$, $\e>0$.
We let $\f_{t,\e}$ denote the solution of the corresponding scalar parabolic flow, with fixed initial
data $\f_0$ (which is a $(\theta+\e \omega)$-psh function with zero Lelong numbers).

Our first observation is that $\e \mapsto \f_{t,\e}$ is non-increasing.  This follows  from the maximum principle, since for $0<\e<\e'$, the function $x \mapsto \f_{\e,t}(x)$ is both
$(\theta+(\e+t) \omega)$-psh and $(\theta+(\e'+t) \omega)$-psh, and satisfies
$$
(\theta+(\e+t) \omega+dd^c \f_{\e,t})^n=e^{\dot{\f}_{\e,t}+h} \omega^n
\leq (\theta+(\e'+t) \omega+dd^c \f_{\e,t})^n,
$$
thus $\f_{\e,t}$ is a subsolution of the parabolic Cauchy problem solved by $\f_{\e',t}$.

We can thus consider
$$
\f_t:=\lim_{\e \rightarrow 0^+} \searrow \f_{\e,t}
$$
and we now need to show that $\f_t$ is bounded from below (so as to guarantee in particular 
that $\f_t \neq -\infty$).
We can assume without loss of generality that $\theta+\omega$ is a K\"ahler form (up to rescaling 
$\omega$). Let $u$ be a continuous $(\theta+\omega)$-psh function such that
$$
(\theta+\omega+dd^c u)^n =e^{u- \f_0} e^h \omega^n,
$$
and consider $u_t=(1-t)\f_0+t u +n (t \log t -t)$. The latter is a family of 
$(\theta+t \omega)$-psh functions such that
$$
(\theta+t\omega+dd^c u_t)^n \geq t^n (\theta+\omega+dd^c u)^n 
=e^{\dot{u}_t+h} \omega^n
$$
with $u_0=\f_0$. It therefore follows from the maximum principle that
$\f_t \geq u_t$ which yields the desired lower bound.

We now briefly indicate how to deal with the key upper bound on $\dot{\f}_t$. 
For $\delta>0$ we fix $\omega_\delta$ a K\"ahler form representing the K\"ahler class
$\a_0+\delta \{\eta\}-\delta c_1(X)$ and let $h_\delta$ be a smooth function such that
$$
\theta+\delta \omega=\omega_\delta+dd^c h_\delta.
$$
Observe now that the parabolic equation can be rewritten, for $t \geq \delta$, as
$$
(\omega_\delta+(t-\delta)\omega+dd^c (\f_t+h_\delta))^n=F_t \omega_\delta^n
$$
where the densities
$$
F_t=e^{\dot{\f}_t+h} \frac{\omega^n}{(\omega_\delta+(t-\delta) \omega)^n}
$$
are uniformly in $L^2$, since (as we have shown earlier)
$$
\dot{\f}_t \leq -\frac{\f_0(x)+C}{t}+C.
$$
It therefore follows from Kolodziej's estimate \cite{Kol98} that $\f_t+h_\delta$ is uniformly bounded
for $t \geq \delta$, hence so is $\f_t$.

The remaining estimates are identical to the ones previously established when $\a_0$ is K\"ahler
as the reader is invited to check.
\end{proof}

\begin{rem}
Recall \cite{BEGZ} that a  nef class does not necessarily contain a positive current with zero Lelong numbers.
One can however slowly attenuate the logarithmic singularities along the lines of the results from section \ref{sec:concluding}.
\end{rem}

\subsection{Deforming along  big and semi-positive classes}

\subsubsection{Continuous initial data}
Recall that a $(1,1)$-class $\a \in H^{1,1}(X,\R)$ is \emph{big} iff it can be represented by a 
\emph{K\"ahler current}, \ie a closed $(1,1)$-current $T$ such that
$T\ge c\om$ for some $c>0$.

\begin{lem} \label{lem:big}
if $\theta$ is a closed real $(1,1)$-form on $X$ whose cohomology class in $H^{1,1}(X,\R)$ is big
then there exists a $\theta$-psh function $\psi_\theta\le 0$ such that:
\begin{itemize}
\item[(i)] $\psi_\theta$ is of class $C^\infty$ on a Zariski open set $\Omega\subset X$, 
\item[(ii)] $\psi_\theta\to-\infty$ near $\partial\Omega$, 
\item[(iii)] $\om_\Omega:=\left(\theta+dd^c\p_\theta\right)|_\Omega$ is the restriction to $\Omega$ of a K\"ahler form on a compactification $\tX$ of $\Omega$ dominating $X$. 
\end{itemize}
\end{lem}

This follows from Demailly's regularization theorem \cite{Dem92}.
Condition (iii) means that there exists a compact K\"ahler manifold $(\tX,\om_{\tX})$ and a modification 
$\pi:\tX\to X$ s.t. $\pi$ is an isomorphism over $\Omega$ and $\pi^*\om_\Omega=\om_{\tX}$ on $\pi^{-1}(\Omega)$. 

By the Noetherian property of closed analytic subsets, the set of all Zariski open subsets $\Omega$ so obtained admits a largest element, called the \emph{ample locus} of $\theta$ and denoted by $\Amp(\theta)$ (see \cite[Theorem 3.17]{Bou04}). Note that $\Amp(\theta)$ only depends on the cohomology class of $\theta$. 

\medskip

 Our starting point is the following result proved in \cite{BG13}, which is a mild generalization of the technical heart of \cite{ST}: 

\begin{thm}\label{thm:ST} 
Let $X$ be a compact K\"ahler manifold, $T\in(0,+\infty)$, and let $(\theta_t)_{t\in[0,T]}$ be a smooth path of closed semipositive $(1,1)$-forms such that $\theta_t\ge\theta$ for a fixed semipositive $(1,1)$-form $\theta$ with big cohomology class. Consider
$$
\mu=e^{\psi^+-\psi^-}\om_X^n
$$
a positive measure on $X$, where
\begin{itemize}
\item $\p^\pm$ are quasi-psh functions on $X$;
\item $e^{-\p^-}\in L^p$ for some $p>1$;
\item $\p^\pm$ are smooth on a given Zariski open subset $U\subset\Amp(\theta)$.
\end{itemize} 
For each continuous $\theta_0$-psh function $\f_0\in C^0(X)\cap\psh(X,\theta_0)$,  there exists a unique bounded continuous function
$\f\in C^0_b\left(U\times[0,T)\right)$ with $\f|_{U\times\{0\}}=\f_0$ and such that on $U\times(0,T)$ $\f$ is smooth and satisfies
\begin{equation}\label{equ:paramain}
\frac{\partial\f}{\partial t}=\log\left[\frac{\left(\theta_t+dd^c\f\right)^n}{\mu}\right]. 
\end{equation}
Furthermore, $\f$ is in fact smooth up to time $T$, \ie $\f\in C^\infty\left(U\times(0,T]\right)$. 
\end{thm}

Our goal is now to allow more singular $\theta$-plurisubharmonic  initial data $\f_0$. For simplicity we only discuss the case of initial data with zero Lelong numbers and we only deal with initial cohomology classes
that are both big and semi-positive.

\subsubsection{General case}

We are going to  prove the following

\begin{thm} \label{thm:big}
Let $X$ be a compact K\"ahler manifold, $T\in(0,+\infty)$, and let $(\theta_t)_{t\in[0,T]}$ be an
affine path of closed semipositive $(1,1)$-forms such that $\theta_t\ge\theta$ for a fixed semipositive $(1,1)$-form $\theta$ with big cohomology class. Consider
$$
\mu=e^{\psi^+-\psi^-}\om^n
$$
a positive measure on $X$, where
\begin{itemize}
\item $\p^\pm$ are quasi-psh functions on $X$;
\item $e^{-\p^-}\in L^p$ for some $p>1$;
\item $\p^\pm$ are smooth on a given Zariski open subset $U\subset\Amp(\theta)$.
\end{itemize} 
For each $\theta_0$-psh function $\f_0\in \psh(X,\theta_0)$ with zero Lelong numbers,  
there exists a unique maximally stretched function
$\f \in C^\infty \left(U\times(0,T)\right)$ such that
\begin{itemize}
\item for all $t \in [0,T]$, $\f_t:=\f(t,\cdot)$ extends as a $\theta_t$-psh function on $X$;
\item  $\| \f_t-\f_0\|_{L^1} \longrightarrow 0$ as $t \rightarrow 0$;
\item $\frac{\partial\f}{\partial t}=\log\left[\frac{\left(\theta_t+dd^c\f\right)^n}{\mu}\right]
\text{ on } U\times(0,T). $
\end{itemize} 
\end{thm}

\begin{proof}
It follows from \cite{EGZ11} that one can find a sequence of continuous 
$\theta_0$-psh functions $\f_{0,j}$ decreasing towards $\f_0$. Let $\f_{t,j}$ denote the 
corresponding approximating solutions provided by Theorem \ref{thm:ST}. It follows again from the 
maximum principle that $j \mapsto \f_{t,j}$ is non increasing, hence we consider
$$
\f_t:=\lim_{ j \rightarrow +\infty} \searrow \f_{t,j}.
$$

The goal is thus to establish a priori estimates. For simplicity we assume $U=\Omega=\Amp(\theta)$ 
(the general case necessitates minor adjustments, see for example \cite[Lemma 3.3.2]{BG13}).
We work on the approximants $\f_{t,j}$, but we suppress the subscript $j$ in the sequel.

We first find an appropriate subsolution. This will guarantee that $\f_t$ is not identically $-\infty$, hence is a well defined $\theta_t$-psh function. The proof is completely similar to the K\"ahler case:  let
$u$ denote a $\theta$-psh function with minimal singularities such that
$$
(\theta+dd^c u)^n=e^{u-\f_0} \mu=e^{u-\f_0+\p^+-\p^-} \omega^n
\; \; \; \text{ in } \; U
$$
(the existence of $u$ follows from \cite{BEGZ}) and consider
$$
u_t:=(1-t) \f_0+t u + n (t \log t -t).
$$
Using that $t \mapsto \theta_t=\theta_0+t \chi$ is affine and $\theta_t \geq \theta$, we observe that 
$$
\theta_t+dd^c u_t=(1-t) (\theta_0+dd^c \f_0)+t(\theta_1+dd^c u) 
\geq t \theta_u \geq 0,
$$
thus $u_t$ is $\theta_t$-psh and 
the reader can check that $u_t$ is a subsolution to the parabolic equation with initial data
$u_0=\f_0$, hence
$$
\f_t \geq (1-t) \f_0+t u + n (t \log t -t),
$$
as desired.

We now establish the key upper-bound on $\dot{\f}_t$ in the context of big cohomology classes.
Fix $\e>0$ and consider
$$
H(t,x)=t \dot{\f}_t-[\f_t-(1-\e)\f_0-\e\p_\theta]-nt,
$$
where $\p_\theta$ is a $\theta$-psh function provided by Lemma \ref{lem:big}.
Since $\dot{\theta}_t=\chi$, a
straightforward computation yields, setting $\omega_t=\theta_t+dd^c \f_t$
and $\Delta_t:=\Delta_{\omega_t}$,
\begin{eqnarray*}
\left(\frac{\partial}{\partial t}-\Delta_t \right)H&=&
t \tr_{\omega_t}(\chi)+\Delta_t(\f_t-(1-\e)\f_0-\e \p_\theta)-n  \\
&=&  -\tr_{\omega_t}((1-\e) T_0+\e (\theta+dd^c \p_\theta) \leq 0,
\end{eqnarray*}
where these estimates are performed in $U=\Omega$.
Since $\p_\theta \rightarrow -\infty$ on $\partial \Omega$, 
It follows that $H$ attains its maximum in $(t=0) \cap \Omega$. Now $H(0,\cdot) \equiv 0$ hence
letting $\e \rightarrow 0$, we obtain
$$
\dot{\f}_t(x) \leq \frac{-\f_0(x)+C}{t}+n.
$$
From there on, one can proceed as in the proof of \cite[Theorem 3.3.4]{BG13} to conclude.
\end{proof}

\subsection{Smoothing properties on midly singular varieties}

\subsubsection{Log terminal singularities}\label{sec:klt}

We assume in this section that $X$ is a \emph{$\Q$-Gorenstein} space, i.e. 
 it has normal singularities and  its canonical bundle $K_X$ exists as a $\Q$-line bundle
($\exists r\in\N$ and a line bundle $L$ on $X$ s.t. $L|_{X_\reg}=rK_{X_\reg}$).

  Choose a \emph{log resolution} of $X$, \ie a projective bimeromorphic morphism $\pi:X'\to X$ which is an isomorphism over $X_\reg$ and whose exceptional divisor $E=\sum_i E_i$ has simple normal crossings. There is a unique collection of rational numbers $a_i$ (the \emph{discrepancies} of $X$ w.r.t the chosen log resolution) s.t.
$$
K_{X'}\sim_\Q\pi^*K_X+\sum_i a_i E_i,
$$
where $\sim_\Q$ denotes $\Q$-linear equivalence. 

\begin{defi}
We say that $X$ has \emph{log terminal singularities} iff $\forall i, \; a_i>-1$. 
\end{defi}

This definition is independent of the choice of a log resolution.  
Quotient singularities are log terminal, and conversely every two-dimensional log terminal singularity is a quotient singularity. Note however that ordinary double points are log-terminal but not quotient singularities in dimension $n \geq 3$.

Choose a local generator $\sigma$ of the line bundle $rK_X$ for some $r\in\N^*$. Restricting to $X_\reg$, we define a smooth positive volume form by setting
\begin{equation}\label{equ:adapted}
\mu_\sigma:=\left(i^{r n^2}\sigma\wedge\bar\sigma\right)^{1/r}. 
\end{equation}
Such measures are called \emph{adapted measures} in \cite{EGZ09}, where the following analytic interpretation of the discrepancies is observed: 

\begin{lem}\label{lem:adapted} 
Let $z_i$ be a local equation of $E_i$, defined on a neighborhood $U\subset X'$ of a given point of $E$. Then we have
$$
\left(\pi^*\mu_\sigma\right)_{U\setminus E}=\prod_i |z_i|^{2a_i}dV
$$
for some smooth volume form $dV$ on $U$. 
\end{lem}

The proof is a direct consequence of the change of variable formula. 
Thus a $\Q$-Gorenstein variety $X$ has log terminal singularities iff every 
adapted measure $\mu_\sigma$ has locally finite mass near each singular point of $X$. 

The construction of adapted measures can be globalized as follows: let $\phi$ be a smooth metric on the $\Q$-line bundle $K_X$. Then 
\begin{equation}\label{equ:adapted2}
\mu_\phi:=\left(\frac{i^{r n^2}\sigma\wedge\bar\sigma}{|\sigma|_{r\phi}}\right)^{1/r}
\end{equation}
becomes independent of the choice of a local generator $\sigma$ of $rK_X$, and hence defines a smooth positive volume form on $X_\reg$, which has locally finite mass near points of $X_\mathrm{sing}$ iff $X$ is log terminal.

\subsubsection{The K\"ahler-Ricci flow on lt varieties}

Starting from a compact K\"ahler manifold $(X_0,\omega_0)$, it is tempting to run the K\"ahler-Ricci flow
until $T_{max}$ and expect that $(X,\omega_t)$ converges, as $t \rightarrow T_{max}$
 towards some mildly singular model $X$ equipped with a limiting current $T_0$ which is not too singular as well, and then try and run again the K\"ahler-Ricci flow from $(X,T_0)$.

When $X$ is projective, there are some evidence that this may be feasible, provided both by algebraic geometry (notably \cite{BCHM}) and by recent progresses in K\"ahler-Ricci flow techniques (see the survey \cite{SW13}).

\smallskip

The main application of our previous technical Theorem \ref{thm:big} is the following extension of a 
 result of Song-Tian \cite{ST}, which shows that one can indeed run the K\"ahler-Ricci flow on a log-terminal variety, starting from an arbitrary positive closed current with zero Lelong numbers:

\begin{thm} 
Let $X$ be a projective complex variety with log terminal singularities. Let $T_0$
be a positive $(1,1)$-current with zero Lelong numbers representing a K\"ahler class 
$\a_0 \in H^{1,1}(X,\R)$. Then there exists a 
continuous family $(\om_t)_{t\in[0,T_{max}[}$ of positive $(1,1)$-currents such that
\begin{itemize} 
\item[(i)] $[\om_t]=\a_0-tc_1(X)$ in $H^{1,1}(X,\R)$; 
\item[(ii)] $\omega_t \rightarrow T_0$ as $t \rightarrow 0$;
\item[(iii)] $(\om_t)_{t\in(0,+\infty)}$ restricts to a smooth path of K\"ahler forms on $X_\reg$ satisfying
$$
\frac{\partial\om_t}{\partial t}=-\Ric(\om_t). 
$$
\end{itemize}
\end{thm}

\begin{proof} 
Let $\pi:X' \rightarrow X$ be a log resolution of $X$. Since $\a_0'=\pi^*\a_0$ is the pull-back of a
K\"ahler class, it is big and semi-positive. We fix $\theta_0$ a smooth closed semi-positive
$(1,1)$-form representing it. Fix $0<T<T_{max}$ and let $\theta_T$ be a smooth closed semi-positive
$(1,1)$-form representing $\pi^*(\a_0-T c_1(X))$. Set
$$
\chi:=\frac{\theta_T-\theta_0}{T}
\; \; \text{ and } \; \; 
\theta_t:=\theta_0+t \chi.
$$
This yields an affine path of big and semi-positive forms. Note that both $\theta_0$ and $\theta_T$ can be chosen as the pull-back by $\pi$ of a K\"ahler form on $X$, hence we can moreover assume that there
is a fixed big and semi-positive form $\theta$ such that $\theta_0,\theta_T \geq \theta$
hence $\theta_t \geq \theta$ for all  $t \in [0,T]$.

Since $\chi$ is a representative of $\pi^*c_1(K_X)$, we can find a smooth metric $\phi$ of $K_X$
whose curvature form is $\pi_*\chi$. We let $\mu_\phi$ denote the corresponding adapted
measure and set
$$
\mu:=\pi^* \mu_\phi=e^{\p^+ - \p^-} \omega_{X'}^n,
$$
where $\omega_{X'}$ denotes a K\"ahler form on $X'$ and $\p^\pm$ are quasi-psh functions
which satisfy the hypotheses of Theorem \ref{thm:big}, as the reader will check 
using Lemma \ref{lem:adapted}  (the Zariski open set $U$ coincides here with $\pi^{-1}(X_{reg})$.
We set 
$$
\pi^* \omega_t=\theta_t+dd^c \f_t
$$
and observe that 
$$
\frac{\partial\om_t}{\partial t}=-\Ric(\om_t)
\Longleftrightarrow 
\frac{\partial\f}{\partial t}=\log\left[\frac{\left(\theta_t+dd^c\f\right)^n}{\mu}\right]
$$
if the potentials $\f_t$ are conveniently normalized. 
The result therefore follows from Theorem \ref{thm:big}.
\end{proof}


\begin{thebibliography}{widestlabel}

\addcontentsline{toc}{chapter}{Bibliography}






\bibitem [Auv11]{Auvray} H.Auvray: The space of Poincar\'e type K\"ahler metrics on the complement of a divisor. Preprint arXiv:1110.4548.



\bibitem[BM87]{BM87} S.~Bando, T.~Mabuchi: Uniqueness of Einstein K{\"a}hler metrics modulo connected group actions, in \emph{Algebraic geometry, Sendai, 1985} (T.~Oda, Ed.), Adv. Stud. Pure Math. {\bf 10}, Kinokuniya, 1987, 11-40. 

\bibitem [BT82] {BT} E.~Bedford, B.~A.~Taylor: A new capacity for plurisubharmonic functions. Acta Math. {\bf 149} (1982), no. 1-2, 1--40.






\bibitem [Ber12]{Ber} R.~Berman: K-polystability of Q-Fano varieties admitting Kaehler-Einstein metrics.
 Preprint arXiv 1205.6214
    


\bibitem [BBGZ13] {BBGZ} R.~Berman, S.~Boucksom, V.~Guedj, A.~Zeriahi: A variational approach to complex Monge-Amp\`ere equations. Publications I.H.E.S. (2013) Vol.{\bf 117}, Issue 1,179-245.

\bibitem [BBEGZ11] {BBEGZ} R.~Berman, S.~Boucksom, P.Eyssidieux, V.~Guedj, A.~Zeriahi: K\"ahler-Ricci flow and Ricci iteration on log-Fano varieties. Preprint arXiv 1111.7158.






\bibitem[Bern13]{Bern13} B.~Berndtsson: The openness conjecture for plurisubharmonic functions.
 Preprint arXiv:1305.5781.


\bibitem[BCHM10]{BCHM} C.~Birkar, P.~Cascini, C.~Hacon, J.~McKernan: Existence of minimal models for varieties of log general type.  J. Amer. Math. Soc. {\bf 23} (2010), no. 2, 405-468.



\bibitem [BK07] {BK07} Z.~B\l{}ocki, S.~Ko\l{}odziej: On regularization of plurisubharmonic functions on manifolds.  Proc. Amer. Math. Soc.  {\bf 135}  (2007),  no. 7, 2089--2093.






 
\bibitem [Bou04]{Bou04} S.~Boucksom: Divisorial Zariski decompositions on compact complex manifolds.  Ann. Sci. Ecole Norm. Sup. (4)  {\bf 37}  (2004),  no. 1, 45--76. 


\bibitem [BEGZ10]{BEGZ} S.~Boucksom, P.~Eyssidieux, V.~Guedj, A.~Zeriahi: Monge-Amp{\`e}re equations in big cohomology classes. Acta Math. {\bf 205} (2010), 199--262. 


\bibitem [BG13]{BG13} S.~Boucksom, V.~Guedj: Regularizing properties of the K\"ahler-Ricci flow. 
{\it An introduction to the K\"ahler-Ricci flow.} Lecture Notes in Math., {\bf 2086}, Springer, Heidelberg, 2013.



\bibitem [Cao85]{Cao85} H.D.Cao: Deformation of K\"ahler metrics to K\"ahler-Einstein metrics on compact K\"ahler manifolds. Invent. Math. {\bf 81} (1985), no. 2, 359-372.





\bibitem[CD07]{CD07} X.~X.~Chen, W.Ding: Ricci flow on surfaces with degenerate initial metrics. 
J. Partial Differential Equations {\bf 20} (2007), no. 3, 193-202. 


\bibitem[CDS12]{CDS} X.~X.~Chen, S.Donaldson, S.Sun: Kahler-Einstein metrics and stability. 
 Preprint arXiv 1210.7494.

\bibitem[CT08]{CT08} X.~X.~Chen, G.Tian: Geometry of K\"ahler metrics and foliations by holomorphic discs.
  Publ. Math. Inst. Hautes Etudes Sci. No. {\bf 107} (2008), 1-107. 

\bibitem[CTZ11]{CTZ11} X.~X.~Chen, G.Tian, Z.Zhang: On the weak K\"ahler-Ricci flow. 
Trans. Amer. Math. Soc. {\bf 363} (2011), no. 6, 2849-2863.
 

\bibitem [CSz12]{CSz12} T.C.Collins, G.Sz\'ekelyhidi: The twisted K\"ahler-Ricci flow. 
 Preprint arXiv:1207.5441 




 


 \bibitem [Dem92]{Dem92} J.~P.~Demailly: Regularization of closed positive currents and intersection theory.  J. Alg. Geom.  {\bf 1}  (1992),  no. 3, 361--409.



\bibitem [DK01]{DK01} J.~P.~Demailly, J.~Koll\'ar: Semi-continuity of complex singularity exponents and K\"ahler-Einstein metrics on Fano orbifolds. Ann. Sci. \'Ecole Norm. Sup. (4) {\bf 34} (2001), no. 4, 525--556. 











 


\bibitem[Don11]{Don} S.~K.~Donaldson: K\"ahler metrics with cone singularities along a divisor. 
 Preprint arXiv 1102.1196.



\bibitem [EGZ08] {EGZ08} P.~Eyssidieux, V.~Guedj, A.~Zeriahi:  A
priori $L^{\infty}$-estimates for degenerate complex Monge-Amp{\`e}re
equations. International Mathematical Research Notes, Vol. {\bf 2008}, Article ID rnn070, 8 pages.
 
\bibitem [EGZ09] {EGZ09} P.~Eyssidieux, V.~Guedj, A.~Zeriahi: Singular K\"ahler-Einstein metrics. J. Amer. Math. Soc. {\bf 22} (2009), 607-639.  

\bibitem [EGZ11] {EGZ11} P.~Eyssidieux, V.~Guedj, A.~Zeriahi: Viscosity solutions to degenerate Complex Monge-Amp\`ere equations. Comm.Pure \& Appl.Math  {\bf 64}(2011), 1059--1094.
  
\bibitem [FJ05]{FJ05} C.Favre, M.Jonsson: Valuations and multiplier ideals.
      J. Amer. Math. Soc. {\bf 18} (2005), no. 3, 655-684.




\bibitem[GT11]{GT11} G.Giesen, P.Topping: Existence of Ricci flows of incomplete surfaces.
  Comm. Partial Differential Equations, {\bf 36} (2011) 1860-1880.


\bibitem [GZ05] {GZ05} V.~Guedj, A.~Zeriahi: Intrinsic capacities on compact K{\"a}hler manifolds. J. Geom. Anal.  {\bf 15}  (2005),  no. 4, 607-639.

\bibitem [GZ07] {GZ07} V.~Guedj, A.~Zeriahi: The weighted Monge-Amp{\`e}re energy
   of quasiplurisubharmonic functions. J. Funct. An.  {\bf 250} (2007), 442-482.


    





\bibitem [Ko\l{}98] {Kol98} S.~Ko\l{}odziej: The complex Monge-Amp{\`e}re equation. Acta Math. {\bf 180} (1998), no. 1, 69--117.










  








\bibitem[PSSW08]{PSSW08} D.~H.~Phong, J.~Song, J.~Sturm, B.~Weinkove: The Moser-Trudinger inequality on K{\"a}hler-Einstein manifolds.  Amer. J. Math.  {\bf 130}  (2008),  no. 4, 1067--1085.










\bibitem [ShW11]{ShW11} M.Sherman, B.Weinkove: Interior derivative estimates for the K\"ahler-Ricci flow.
 Preprint arXiv:1107.1853, to appear in Pacific J. Math.



\bibitem [Siu87]{Siu} Y.~T.~Siu: Lectures on Hermitian-Einstein metrics for stable bundles and K\"ahler-Einstein metrics. DMV Seminar, 8. Birkh\"auser Verlag, Basel, 1987. 



\bibitem [Sko72]{Sko} H.~Skoda: Sous-ensembles analytiques d'ordre fini ou infini dans  $\C^{n}$. Bull. Soc. Math. France  {\bf 100}  (1972), 353--408. 

\bibitem [ST12]{ST12} J.~Song, G.~Tian: Canonical measures and K\"ahler-Ricci flow. 
  J. Amer. Math. Soc. {\bf 25} (2012), no. 2, 303-353.

\bibitem [ST09]{ST} J.~Song, G.~Tian: The K\"ahler-Ricci flow through singularities. Preprint (2009) arXiv:0909.4898. 

\bibitem [SW13]{SW13} J.~Song, B.Weinkove: An introduction to the K\"ahler-Ricci flow.
{\it An introduction to the K\"ahler-Ricci flow.} Lecture Notes in Math., {\bf 2086}, Springer, Heidelberg, 2013.



\bibitem[SzTo11]{SzTo} G.~Sz\'ekelyhidi, V.~Tosatti: Regularity of weak solutions of a complex Monge-Amp\`ere equation. Anal. PDE {\bf 4} (2011), no. 3, 369-378. 


\bibitem [R12]{R12} T.Richard:  Canonical smoothing of compact Alexandrov surfaces via Ricci flow.
 Preprint arXiv:1204.5461.

\bibitem[Tian97]{Tian97} G.~Tian: K{\"a}hler-Einstein metrics with positive scalar curvature. Inv. Math. {\bf 130} (1997), 239--265.

\bibitem [Tian]{Tian} G.~Tian: Canonical metrics in K{\"a}hler geometry. Lectures in Mathematics ETH Z{\"u}rich. Birkh{\"a}user Verlag, Basel (2000).

\bibitem [Tian12]{Tian12} G.~Tian: K-stability and Kaehler-Einstein metrics.
 Preprint arXiv 1211.4669.

\bibitem [TZha06]{TZha} G.~Tian, Z.Zhang: On the K\"ahler-Ricci flow on projective manifolds of general type. Chinese Ann. Math. Ser. B {\bf 27} (2006), no. 2, 179-192.

\bibitem[TZ07]{TZ07} G.~Tian, X.~Zhu: Convergence of K\"ahler-Ricci flow. J. Amer. Math. Soc. {\bf 20} (2007), no. 3, 675--699.

\bibitem [Top10]{Top10} P.Topping: Ricci flow compactness via pseudolocality, and flows with incomplete initial metrics. J.E.M.S., {\bf 12} (2010) 1429-1451.

\bibitem [Top12]{Top12} P.Topping: Uniqueness and nonuniqueness for Ricci flow on surfaces: reverse cusp singularities. Int. Math. Res. Not. IMRN (2012), no. {\bf 10}, 2356-2376. 

 

\bibitem[Tsu88]{Tsu} H.~Tsuji: Existence and degeneration of K{\"a}hler-Einstein metrics on minimal algebraic varieties of general type.  Math. Ann. {\bf  281}  (1988),  no. 1, 123--133.



\bibitem [Vaz06]{Vaz06} J.L.Vazquez: Smoothing and decay estimates for nonlinear diffusion equations. Equations of porous medium type. Oxford Lecture Series in Mathematics and its Applications, 33. Oxford University Press, Oxford, 2006. xiv+234 pp.
   
 
\bibitem [Yau78]{Yau}  S.~T.~Yau: On the Ricci curvature of a compact K{\"a}hler manifold and the complex Monge-Amp{\`e}re equation. I. Comm. Pure Appl. Math. {\bf 31} (1978), no. 3, 339--411.  
  



 
\end{thebibliography}
\end{document}